\documentclass[12pt,a4paper]{amsart}

\usepackage{amssymb}
\usepackage{amsthm}
\usepackage{amsmath}
\usepackage{graphicx}
\usepackage{enumerate}
\usepackage{xcolor}
\usepackage{tikz}
\usepackage{graphbox}
\usepackage{orcidlink}
\usepackage{mathrsfs}  
\usepackage{hyperref}
\hypersetup{colorlinks=true,allcolors=blue}
\usepackage[margin=2.5cm]{geometry}

\newtheorem{theorem}{Theorem}[section]
\newtheorem{proposition}[theorem]{Proposition}

\newtheorem{lemma}[theorem]{Lemma}

\theoremstyle{definition}
\newtheorem{example}[theorem]{Example}
\newtheorem{definition}[theorem]{Definition}
\newtheorem{remark}[theorem]{Remark}

\newcommand{\II}{\mathbb{I}}
\newcommand{\RR}{\mathbb{R}}

\newcommand{\diag}{\mathrm{diag}}

\newcommand{\uD}{\underline{\Delta}}

\title[Diagonal sections of Markov products]{On diagonal sections of Markov products of copulas}

\author[R. Drnovšek]{Roman Drnovšek \orcidlink{0000-0003-3342-7417} }
\address{University of Ljubljana, Faculty of Mathematics and Physics, and Institute of Mathematics, Physics and Mechanics, Ljubljana, Slovenia}
\email{Roman.Drnovsek@fmf.uni-lj.si}

\author[D. Kokol Bukovšek]{Damjana Kokol Bukovšek \orcidlink{0000-0002-0098-6784} }
\address{University of Ljubljana, School of Economics and Business, and Institute of Mathematics, Physics and Mechanics, Ljubljana, Slovenia}
\email{Damjana.Kokol.Bukovsek@ef.uni-lj.si}

\author[N. Stopar]{Nik Stopar \orcidlink{0000-0002-0004-4957} } 
\address{University of Ljubljana, Faculty of Civil and Geodetic Engineering, University of Ljubljana, Faculty of Mathematics and Physics, and Institute of Mathematics, Physics and Mechanics, Ljubljana, Slovenia}
\email{Nik.Stopar@fgg.uni-lj.si}


\begin{document}

\begin{abstract}
Given a bivariate copula $C$, the bivariate copula $\Psi(C)$ is defined as the Markov product of $C^T$ and $C$. In this paper we consider the diagonal section of $\Psi(C)$ under the additional assumption that its value at a single point is known. We determine the exact upper and lower bounds for the diagonal section of 
$\Psi(C)$, given that $\Psi(C)(a,a)=d \in [a^2,a]$, where $a \in [0,1]$ is a fixed point.
Furthermore, we describe some additional properties of the transformation $\Psi$, such as its behavior under ordinal sums and the survival operation.
We also study checkerboard copulas in the range of $\Psi$ and matrices associated to them.
\end{abstract}

\maketitle

\section{Introduction}

Dependence of continuous random variables is typically modeled with copulas due to Sklar's theorem \cite{Skl59}. Copulas enable the modeling of various types of dependence structures and there exist various well-known parametric families of copulas that are tailored for specific applications \cite{GeNeZi11,GuSe12,JaHuMi17,JoLiNi10} (see also \cite{AnRo24} for a collection of dependence properties of different copula families).
In certain situations it is desirable to quantify the degree of dependence with a single number, which lead to the development of various measures of association, such as Pearson's correlation coefficient, measures of concordance \cite{DoUbFl06,FuSc21,GeNeGh10,JuGrTr21,Sc84,ScPaTr17,UbFl05,Wy15}, measures of dependence \cite{Ch21,DeSiSt13,Li14,Re59,DaEA11,ScWo81}, coefficients of tail dependence \cite{DuFePa15,HuJo14,LeTa96,ScSc07}, etc.

Classical measures of association are typically symmetric, i.e., the measure remains unchanged if we permute the random variables. Since, in practice, one random variable may have a stronger influence on another than vice versa, non-symmetric measures have recently been introduced. Such measures can measure the degree of functional dependence of a single variable $Y$ on a set of explanatory variables $X_1,X_2,\ldots,X_d$.
One such measure is the Chatterjee's rank correlation coefficient $\xi$ introduced in \cite{DeSiSt13,Ch21,AzCh21}. In the bivariate setting, the population version of the coefficient is defined as
$$\xi(Y,X)=\frac{\int_{-\infty}^\infty \operatorname{\mathbb{V}ar}(\mathbb{P}(Y \ge y \mid X)) d\mathbb{P}^Y(y)}{\int_{-\infty}^\infty \operatorname{\mathbb{V}ar}(\mathbf{1}_{\{Y \ge y\}}) d\mathbb{P}^Y(y)}.$$
Chatterjee's $\xi$ has the property that $\xi(Y,X)=0$ if and only if $Y$ is independent of $X$, and $\xi(Y,X)=1$ if and only if there exists a measurable function $f$ such that $Y=f(X)$ almost surely.
The multivariate version of Chatterjee's coefficient can be found in \cite{AzCh21}.

For a continuous random vector $(X,Y)$ the coefficient $\xi(Y,X)$ is uniquely determined by the copula $C$ of $(X,Y)$ and is expressed as
$$\xi(Y,X)=\xi(C)=6\int_0^1\int_0^1 \big(\partial_1 C(u,v)\big)^2 dudv-2.$$
Here, $\partial_1 C(u,v)$ denotes the partial derivative of $C$ with respect to the first variable, which exists for all $v \in [0,1]$ almost surely with respect to $u \in [0,1]$.
Alternatively, $\xi(C)$ can be expressed as (see \cite[Theorem~4]{Fu24})
\begin{equation}\label{eq:main}
    \xi(C)=\phi(\Psi(C)),
\end{equation}
where
$$\phi(D)=6\int_0^1 D(u,u) du-2$$
is the Spearman's footrule of copula $D$ and $\Psi(C)=C^T *C$ is the Markov product of $C^T$ and $C$, see Section~\ref{sec:prelim}.
The transformation $\Psi$ above is a special case (with $d=1$) of a more general dimension reduction principle, which transforms a $(d+1)$-dimensional copula to a bivariate copula, see \cite{Fu24}.
This transformation is interesting since it can be used to translate certain dependence concepts from the multivariate setting to the bivariate setting, see \cite{FuLi26}.
For example, suppose that $(d+1)$-dimensional copula $C$ is the distribution function of a random vector $(U_1,U_2,\ldots,U_d,V)$ and $\Psi(C)$ is the distribution function of $(V,W)$, where $U_i$, $V$ and $W$ are uniformly distributed on $[0,1]$.
Then $V$ perfectly depends on $(U_1,U_2,\ldots,U_d)$ if and only if $V=W$ almost surely, see \cite[Theorem~3.1]{FuLi26}.

In this paper we study the range of transformation $\Psi$ in the bivariate case, i.e., the set $\{\Psi(C) \mid C \text{ a bivariate copula}\}$. Because of the formula in equation \eqref{eq:main} the diagonal sections of copulas in the range of $\Psi$ are of particular interest.
Recall that the family of diagonal sections of all bivariate copulas is easily characterized as the set of all increasing, $2$-Lipschitz functions $\delta \colon [0,1] \to [0,1]$ that satisfy $\delta(0)=0$, $\delta(1)=1$, and $\delta(u) \le u$ for all $u \in [0,1]$, see \cite{FrNe97}.
The diagonal sections of copulas in the range of $\Psi$ are much more restricted.
Their study was initiated in \cite{Fu24}, where the author showed that $\Psi(C)(u,u) \ge u^2$ for all $u \in [0,1]$, so the diagonal section of $\Psi(C)$ always dominates the diagonal section of the product copula $\Pi$ (although $\Psi(C)$ need not dominate $\Pi$ globally).

We focus our attention on the diagonal section of $\Psi(C)$ under the additional assumption that its value at a single point is known.
We determine the exact upper and lower bound for the diagonal section of $\Psi(C)$, given that $\Psi(C)(a,a)=d \in [a^2,a]$, where $a$ is a fixed point in $[0,1]$.
Furthermore, we establish some additional properties of the transformation $\Psi$, such as its behavior under ordinal sums and survival operation.
Since we use discretization methods in our proofs, we also consider checkerboard copulas in the range of $\Psi$ and matrices associated to them, namely, the value matrix, the mass matrix, and the density matrix.

The paper is structured as follows. In Section~\ref{sec:prelim} we recall the notions of Markov product, ordinal sum of copulas, checkerboard copulas, and completely positive matrices.
Section~\ref{sec:matrices} is devoted to matrices associated to checkerboard copulas. Some properties of the transformation $\Psi$, along with the properties of matrices associated to checkerboard copulas in the range of $\Psi$, are given in Section~\ref{sec:properties}.
In Sections~\ref{sec:lower} and \ref{sec:upper} we determine the exact lower and upper bound for the diagonal section of a copula in the range of transformation $\Psi$ under the condition that its value at a single point is known.
We illustrate our results with a couple of examples and end the paper with some concluding remarks.

\section{Preliminaries} \label{sec:prelim}

Let $\II$ be the unit interval $[0,1]\subseteq\RR$ and $u_1, u_2, v_1, v_2 \in \II$ be such that $u_1 \le u_2$ and $v_1 \le v_2$. The Cartesian product of intervals $R=[u_1,u_2]\times [v_1,v_2]$ is called a {\it rectangle} in $\II^2$. Let $H:\II^2 \to \RR$ be a real function. We define the $H$-volume of $R$ as
$$V_H(R)=H(u_2,v_2)-H(u_2,v_1)-H(u_1,v_2)+H(u_1,v_1).$$
A {\it bivariate copula} is a function $C:\II^2 \to \II$ with the following properties:
\begin{enumerate}[(i)]
    \item $C(0,v)=C(u,0)=0$ for all $u,v \in \II$ ($C$ is {\it grounded}),
    \item $C(u,1)=u$ and $C(1,v)=v$ for all $u,v \in \II$ ($C$ has {\it uniform marginals}), and
    \item $V_C(R) \geq 0$ for every rectangle $R \subseteq \II^2$ ($C$ is {\it $2-$increasing}).
\end{enumerate}

Recall that the lower and upper bound for bivariate copulas are the Fr\'{e}chet-Hoeffding bounds
$$W(u,v)=\max \{0,u+v-1\} \qquad \text{and} \qquad M(u,v)=\min \{u,v \}.$$
The survival copula $\widehat{C}$ of a copula $C$ is defined as
\begin{equation}\label{eq:Chat}
\begin{aligned}
    \widehat{C}(u, v) &= C(1,1)- C(1-u,1)-C(1,1-v) + C(1-u, 1-v) = \\
    &= 1 - (1-u)-(1-v) + C(1-u, 1-v) \\
    &= u+v-1 + C(1-u, 1-v).
\end{aligned}
\end{equation}

Given two bivariate copulas $C$ and $D$, their \emph{Markov product} (also called \emph{star product}), defined in \cite{DarNguOls92}, is given for all $(u,v) \in \II^2$ as
\begin{equation}\label{eq:markov}
    (C \ast D)(u,v) = \int_0^1 \partial_2 C(u,t) \partial_1 D(t,v) dt.
\end{equation}
This is well defined, since copulas are Lipschitz continuous, hence, their partial derivatives exist almost everywhere. Furthermore, $C \ast D$ is again a copula, see \cite[Theorem~2.1]{DarNguOls92}. In \cite{Fu24} the author introduces a transformation $\Psi$ on the set of copulas defined as
$\Psi(C)=C^T \ast C$, or equivalently,
\begin{equation}\label{eq:psi}
    \Psi(C)(u,v) = \int_0^1 \partial_1 C(t,u) \partial_1 C(t,v) dt
\end{equation}
for all $(u,v) \in \II^2$. Note that $\Psi(C)$ is a symmetric copula, i.e., $\Psi(C)(u,v)=\Psi(C)(v,u)$.

For a finite or countably infinite family of disjoint nonempty intervals $\{J_i=(a_i,b_i)\}_{i \in I}$ and a family of copulas $\{C_i\}_{i \in I}$, the \emph{ordinal sum} $C$ of copulas $\{C_i\}_{i \in I}$ with respect to $\{J_i\}_{i \in I}$ is given by
\begin{equation}\label{eq:ordinal}
    C(u,v)=\begin{cases}
        a_i+(b_i-a_i)C_i\big(\frac{u-a_i}{b_i-a_i},\frac{v-a_i}{b_i-a_i}\big), & u,v \in (a_i,b_i) \text{ for some } i \in I,\\
        M(u,v), & \text{otherwise}.
    \end{cases}
\end{equation}
The ordinal sum is a copula, whose mass distribution inside the square $[a_i,b_i] \times [a_i,b_i]$ is a rescaled version of the mass distribution of copula $C_i$, and the mass distribution outside the union of squares $\{[a_i,b_i] \times [a_i,b_i] \mid i \in I\}$ coincides with the mass distribution of copula $M$.

For any positive integer $n$, the set $\{1, 2, \ldots, n\}$ is abbreviated by $[n]$. For positive integers $m, n$ and splitting points $0 = a_0 < a_1 < a_2 < \ldots < a_{m-1} < a_m = 1$ and $0 = b_0 < b_1 < b_2 < \ldots < b_{n-1} < b_n = 1$, we define the sets 
\begin{align*}
\delta_1 &:= \{a_0, a_1, a_2, \ldots, a_{m-1}, a_m\}  \subseteq \II;\\
\delta_2 &:= \{b_0, b_1, b_2, \ldots, b_{n-1}, b_n\}  \subseteq \II;\\
\Delta &:= \delta_1 \times \delta_2 \subseteq \II^2.
\end{align*}
We will refer to sets of the type $\Delta$ as \emph{discrete meshes}. If $\delta_1 = \delta_2$, we will denote the mesh $\delta_2 \times \delta_2$ by $\uD$ and call it a \emph{symmetric mesh}.

\begin{definition}
A function  $\widetilde{C}\colon\Delta\to\II$ is called a \emph{discrete copula}, if
\begin{enumerate}[(i)]
    \item $\widetilde{C}(a_i, 0)=\widetilde{C}(0, b_j)=0$ for all $i \in [m], j \in [n]$ ($\widetilde{C}$ is \emph{grounded}),
    \item $\widetilde{C}(a_i, 1)=a_i$ and $\widetilde{C}(1, b_j)=b_j$ for all $i \in [m], j \in [n]$ ($\widetilde{C}$ has \emph{uniform marginals}), and
    \item  $V_{\widetilde{C}}([a_{i-1},a_i]\times[b_{j-1},b_j])  \ge 0$ ($\widetilde{C}$ is \emph{$2$-increasing}).
\end{enumerate}
\end{definition}

Any discrete copula $\widetilde{C}$ can be extended to a copula $A$ by using bilinear functions. This (piecewise) bilinear extension is often referred to as a \emph{checkerboard copula} or \emph{checkerboard extension}, see \cite{MikTay10,LinWanZhaZha25}. For any $(u,v) \in \II^2$ there exist $i \in [m], j \in [n]$ such that 
$(u,v) \in [a_{i-1},a_i]\times[b_{j-1},b_j]$. Let 
$t = \tfrac{u-a_{i-1}}{a_i-a_{i-1}} \in [0,1]$, and $s = \tfrac{v-b_{j-1}}{b_j-b_{j-1}} \in [0,1]$. The value of copula $A$ at the point $(u,v)$ is then given as
\begin{equation}\label{eq:CB}
    A(u,v) = (1-t)(1-s)\widetilde{C}(a_{i-1},b_{j-1}) + t(1-s)\widetilde{C}(a_i,b_{j-1}) + (1-t)s\widetilde{C}(a_{i-1},b_j) +ts\widetilde{C}(a_i,b_j).
\end{equation}
If a checkerboard copula $A$ is obtained from a discrete copula $\widetilde{C}$ defined on a mesh $\Delta$, we will refer to it as a 
\emph{checkerboard copula with splitting points $\Delta$}. We will typically denote checkerboard copulas by $A$, $B$, $A_1,\ldots$, and general copulas by $C$, $D$, $C_1,\ldots$

We recall also a notion from linear algebra. We denote the vector space of real $n \times m$ matrices by $\RR^{n \times m}$ and the transpose of matrix $M$ by $M^T$. A square real matrix $M \in \RR^{n \times n}$ is called \emph{completely positive}, if there exists an entry-wise nonnegative matrix $N \in \RR^{m\times n}$, such that $M=N^TN$. It is obvious, that every completely positive matrix is entry-wise nonnegative and positive semidefinite. Conversely, 
every $2 \times 2$, $3 \times 3$ and $4 \times 4$ matrix which is  entry-wise nonnegative and positive semidefinite is also completely positive, but there exists a $5 \times 5$ entry-wise nonnegative and positive semidefinite matrix which is not completely positive, see \cite{BeShMo03}.

\section{Matrices associated to checkerboard copulas}\label{sec:matrices}

Let $A$ be a checkerboard copula with splitting points $\Delta$. Since $A$ is bilinear on every rectangle $[a_{i-1},a_i]\times[b_{j-1},b_j]$, its density $\partial_1\partial_2A$ exists for any $(u,v) \in (a_{i-1},a_i)\times(b_{j-1},b_j)$ and it is constant there. We will assign to copula $A$ three matrices in $\RR^{m\times n}$, the value matrix $S_A$, the mass matrix $M_A$, and the density matrix $G_A$, defined as follows.  

\begin{definition}
Let $A$ be a checkerboard copula with splitting points $\Delta$. 
\begin{enumerate}[(i)]
    \item The \emph{value matrix} $S_A = [s_{ij}]_{ij} \in \RR^{m\times n}$ is the matrix with entries $$s_{ij} = A(a_i, b_j).$$
    \item The \emph{mass matrix} $M_A = [m_{ij}]_{ij} \in \RR^{m\times n}$ is the matrix with entries $$m_{ij} = V_A([a_{i-1},a_i]\times[b_{j-1},b_j]).$$
    \item The \emph{density matrix} $G_A = [g_{ij}]_{ij} \in \RR^{m\times n}$ is the matrix with entries $$g_{ij} = \partial_1\partial_2A(\tfrac{a_{i-1}+a_i}{2},\tfrac{b_{j-1}+b_j}{2}).$$
\end{enumerate}
\end{definition}

Clearly, the matrix $S_A$ uniquely determines the splitting points $\Delta$ and the checkerboard copula $A$. The same holds for matrices $M_A$ and $G_A$, see Proposition~\ref{prop:connection} below.

Note that, given a fixed $\Delta$, the value matrix and the mass matrix could be defined for any copula $C$, but the density matrix is defined only for checkerboard copulas. These matrices are closely related. In order to state the relationships between them, we define a few more matrices.
Let $J_n \in \RR^{n\times n}$ be the matrix with $1$'s on the diagonal, $-1$'s on the first superdiagonal and $0$'s elsewhere, and $P_n$ the matrix with $1$'s on the counter-diagonal and $0$'s elsewhere, i.e., 
$$J_n = \begin{bmatrix}
    1 & -1 & 0 & \cdots & 0 & 0 \\
    0 & 1 & -1 & \cdots & 0 & 0 \\
    0 & 0 & 1 & \ddots &  & 0 \\
    \vdots & \vdots & \vdots& \ddots& \ddots& \vdots \\
    0 & 0 & 0 & \cdots & 1 & -1 \\
    0 & 0 & 0 & \cdots & 0 & 1 
\end{bmatrix}, \quad P_n = \begin{bmatrix}
    0 & 0 &  \cdots & 0 & 1 \\
    0 & 0  & \cdots & 1 & 0 \\
    \vdots & \vdots & & \vdots& \vdots \\
    0 & 1 & \cdots & 0 & 0 \\
    1 & 0 & \cdots & 0 & 0 
\end{bmatrix}.$$
By $\mathbf{v}_{\delta_1}\in \RR^{m}$  and $\mathbf{v}_{\delta_2}\in \RR^{n}$ we denote the vectors
\begin{align*}
   \mathbf{v}_{\delta_1} &= (a_1-a_0, a_2-a_1, ..., a_m-a_{m-1})^T, \\
   \mathbf{v}_{\delta_2} &= (b_1-b_0, b_2-b_1, ..., b_n-b_{n-1})^T,
\end{align*}
and by $\mathbf{1}_m\in \RR^{m}$  and $\mathbf{1}_n\in \RR^{n}$ the vectors of all $1$'s.
Furthermore, by $D_{\delta_1}\in \RR^{m\times m}$  and $D_{\delta_2}\in \RR^{n\times n}$  we denote the diagonal matrices
\begin{align*}
   D_{\delta_1}= \diag(\mathbf{v}_{\delta_1}) \quad \text{and} \quad
   D_{\delta_2}= \diag(\mathbf{v}_{\delta_2}).
\end{align*}

In next proposition we give the connection between the value matrix, the mass matrix and the density matrix of a checkerboard copula. Part of this connection has already been given in \cite{KoPe25}.

\begin{proposition}\label{prop:connection}
    Let $A$ be a checkerboard copula with splitting points $\Delta$. Then
$$ M_A = J_m^T S_A J_n = D_{\delta_1} G_A D_{\delta_2}. $$
\end{proposition}

\begin{proof}
    For every $i\in [m], j \in [n]$ we have
\begin{align*}
  m_{ij} &= V_A([a_{i-1},a_i]\times[b_{j-1},b_j]) \\
  &= A(a_i,b_j)-A(a_i,b_{j-1})-A(a_{i-1},b_j)+A(a_{i-1},b_{j-1}) \\
  &= s_{ij}-s_{i,j-1}-s_{i-1,j}+s_{i-1,j-1}, 
\end{align*}
where we adopt the convention $s_{0j}=s_{i0}=0$. It follows that $M_A = J_m^T S_A J_n$. Furthermore, since copula $A$ is absolutely continuous, 
\begin{align*}
  m_{ij} &= V_A([a_{i-1},a_i]\times[b_{j-1},b_j]) = \int_{a_{i-1}}^{a_i} \int_{b_{j-1}}^{b_j} \partial_1\partial_2A(u,v)\,du\,dv\\
  &= \int_{a_{i-1}}^{a_i} \int_{b_{j-1}}^{b_j} g_{ij}\,du\,dv = (a_i-a_{i-1})(b_j-b_{j-1})g_{ij}, 
\end{align*}
thus $M_A = D_{\delta_1} G_A D_{\delta_2}$.
\end{proof}

The following lemma relates the matrices of checkerboard copula $A$ and its survival copula $\widehat{A}$.

\begin{lemma}\label{lem:checkerboard_hat}
    If $A$ is a checkerboard copula with splitting points $\Delta$, then the survival copula $\widehat{A}$ is also a checkerboard copula with splitting points $\Delta' = \delta'_1 \times \delta'_2$, where
\begin{align*}
\delta'_1 &:= \{1-a_m, 1-a_{m-1}, \ldots, 1-a_1, 1-a_0\}  \subseteq \II,\\
\delta'_2 &:= \{1-b_n, 1-b_{n-1}, \ldots, 1-b_1, 1-b_0\}  \subseteq \II,
\end{align*}
and 
 $$ S_{\widehat{A}} = J_m^{-T} P_m J_m^T S_A J_n P_n J_n^{-1}, \qquad M_{\widehat{A}} = P_m M_A P_n, \qquad G_{\widehat{A}} = P_m G_A P_n. $$
\end{lemma}

\begin{proof}
    By equation \eqref{eq:Chat} we have 
$$\partial_1\partial_2\widehat{A}(u,v) = \partial_1\partial_2A(1-u,1-v).$$
so $\partial_1\partial_2\widehat{A}(u,v)$ is constant on every rectangle $(1-a_i,1-a_{i-1})\times (1-b_j,1-b_{j-1})$. Thus $\widehat{A}$ is a checkerboard copula with splitting points $\Delta'$. For its density matrix $G_{\widehat{A}} = [g'_{ij}]_{ij}$ we have
$g'_{ij} = g_{m+1-i,n+1-j}$ for any $i\in [m]$ and $j \in [n]$, so $G_{\widehat{A}} = P_m G_A P_n$. Furthermore,
\begin{align*}
    M_{\widehat{A}} &= D_{\delta'_1} G_{\widehat{A}} D_{\delta'_2} = D_{\delta'_1} P_m G_A P_n D_{\delta'_2} = D_{\delta'_1} P_m D_{\delta_1}^{-1} M_A 
    D_{\delta_2}^{-1} P_n D_{\delta'_2} = P_m M_A P_n, \text{ and} \\
    S_{\widehat{A}} &= J_m^{-T} M_{\widehat{A}} J_n^{-1} = J_m^{-T} P_m M_A P_n J_n^{-1} = J_m^{-T} P_m J_m^T S_A J_n P_n J_n^{-1}
\end{align*}
by Proposition~\ref{prop:connection}.
\end{proof}

We conclude this section by characterizing when a matrix is the mass matrix, the density matrix, or the value matrix of some checkerboard copula.

\begin{proposition}\label{prop:matrices}
Let $\Delta=\delta_1 \times \delta_2$ be a mesh of splitting points, where $|\delta_1|=m+1$ and $|\delta_2|=n+1$.
\begin{enumerate}[$(i)$]
 \item A matrix $M = [m_{ij}]_{ij} \in \RR^{m\times n}$ is the mass matrix of some checkerboard copula with splitting points $\Delta$ if and only if it is entry-wise nonnegative and 
$$ M\mathbf{1}_n = \mathbf{v}_{\delta_1}, \qquad M^T\mathbf{1}_m = \mathbf{v}_{\delta_2}. $$
 \item A matrix $G = [g_{ij}]_{ij} \in \RR^{m\times n}$ is the density matrix of some checkerboard copula with splitting points $\Delta$ if and only it is entry-wise nonnegative and 
 $$ G\mathbf{v}_{\delta_2} = \mathbf{1}_m, \qquad G^T\mathbf{v}_{\delta_1} = \mathbf{1}_n. $$
 \item A matrix $S = [s_{ij}]_{ij} \in \RR^{m\times n}$ is the value matrix of some checkerboard copula with splitting points $\Delta$ if and only if the matrix $J_m^T S J_n$  is entry-wise nonnegative and 
$$ s_{in} = a_i \text{ for all } i \in [m], \text { and } s_{mj} = b_j \text{ for all } j \in [n].$$
\end{enumerate}
\end{proposition}

\begin{proof}
$(i)$: Let $M$ be a mass matrix of a checkerboard copula with splitting points $\Delta$. Then $M$ is entry-wise nonnegative and we have for every $i \in [m]$
$$\sum_{j=1}^n m_{ij} = \sum_{j=1}^n V_A([a_{i-1},a_i] \times [b_{j-1},b_j]) = V_A([a_{i-1},a_i] \times [0, 1]) = a_i-a_{i-1},$$
so $M\mathbf{1}_n = \mathbf{v}_{\delta_1}$. The equality $M^T\mathbf{1}_m = \mathbf{v}_{\delta_2}$ is proved analogously. Now suppose that $M$ is entry-wise nonnegative, $M\mathbf{1}_n = \mathbf{v}_{\delta_1}$ and $M^T\mathbf{1}_m = \mathbf{v}_{\delta_2}$. Note that vectors $\mathbf{v}_{\delta_1}$ and $\mathbf{v}_{\delta_2}$ uniquely determine the sets $\delta_1$ and $\delta_2$ and that $\mathbf{1}_m^T\mathbf{v}_{\delta_1} = a_m = \mathbf{1}_n^T\mathbf{v}_{\delta_2} = b_n = 1$. Let $S = J_m^{-T} M_{\widehat{A}} J_n^{-1}$. For every $i\in [m]$ and $j \in [n]$ we have
$$ s_{ij} = \sum_{k=1}^i \sum_{\ell=1}^j m_{k\ell}.$$
Let $\widetilde{A} \colon \Delta \to \II$ be a function defined by $\widetilde{A}(a_i,b_j) = s_{ij}$ and $\widetilde{A}(0,0) = \widetilde{A}(a_i,0) = \widetilde{A}(0, b_j) = 0$ for every $i \in [m]$ and $j \in [n]$. Then $\widetilde{A}$ is grounded, we have $V_{\widetilde{A}}([a_{i-1},a_i]\times[b_{j-1},b_j]) = m_{ij} \ge 0$, 
$$\widetilde{A}(a_i,1) = s_{in} = \sum_{k=1}^i \sum_{\ell=1}^n m_{k\ell} = \sum_{k=1}^i (a_k-a_{k-1}) = a_i,$$
and, similarly, $\widetilde{A}(1,b_j) = b_j$ for every $i\in [m]$ and $j \in [n]$. Hence, $\widetilde{A}$ is a discrete copula. Finally, we extend discrete copula $\widetilde{A}$ to a checkerboard copula $A$ using bilinear functions. By the above, the mass matrix of copula $A$ is $M$.

$(ii)$: Let $G$ be the density matrix of a checkerboard copula with splitting points $\Delta$. Then $M = D_{\delta_1} G D_{\delta_2}$ is its mass matrix by Proposition~\ref{prop:connection}, $M$ is entry-wise nonnegative, $M\mathbf{1}_n = \mathbf{v}_{\delta_1}$ and $M^T\mathbf{1}_m = \mathbf{v}_{\delta_2}$ by $(i)$. It follows that $G=D_{\delta_1}^{-1} M D_{\delta_2}^{-1}$ is entry-wise nonnegative, 
$$G\mathbf{v}_{\delta_2} = (D_{\delta_1}^{-1} M D_{\delta_2}^{-1})(D_{\delta_2} \mathbf{1}_n) = D_{\delta_1}^{-1} \mathbf{v}_{\delta_2} = \mathbf{1}_n,$$ 
and, similarly, $G^T\mathbf{v}_{\delta_1} = \mathbf{1}_m$. Conversely, suppose that matrix $G$ is entry-wise nonnegative, $G\mathbf{v}_{\delta_2} = \mathbf{1}_n$ and $G^T\mathbf{v}_{\delta_1} = \mathbf{1}_m$. Let $M = D_{\delta_1} G D_{\delta_2}$. Then $M$ is entry-wise nonnegative, $M\mathbf{1}_n = D_{\delta_1} G D_{\delta_2}\mathbf{1}_n = D_{\delta_1} G \mathbf{v}_{\delta_2} = D_{\delta_1} \mathbf{1}_m = \mathbf{v}_{\delta_1}$ and $M^T\mathbf{1}_m = \mathbf{v}_{\delta_2}$, so $M$ is the mass matrix of some checkerboard copula with splitting points $\Delta$ by $(i)$, and thus $G$ is its density matrix.

$(iii)$: Let $S$ be the value matrix of a checkerboard copula with splitting points $\Delta$. Then $M = J_m^T S J_n$ is its mass matrix by Proposition~\ref{prop:connection}, $M$ is entry-wise nonnegative, $M\mathbf{1}_n = \mathbf{v}_{\delta_1}$ and $M^T\mathbf{1}_m = \mathbf{v}_{\delta_2}$ by $(i)$. As already proven in $(i)$, we have $s_{in} = a_i$ and $s_{mj} = b_j$ for every $i\in [m]$ and $j \in [n]$. Conversely, suppose that $S$ is a matrix, such that $M = J_m^T S J_n$ is entry-wise nonnegative, $s_{in} = a_i$ and $s_{mj} = b_j$ for every $i\in [m]$ and $j \in [n]$. Then 
$$M\mathbf{1}_n = J_m^T S J_n \mathbf{1}_n = J_m^T S (0, 0, ..., 0, 1)^T = J_m^T (a_1, a_2, ..., a_m)^T = \mathbf{v}_{\delta_1}$$ and $M^T\mathbf{1}_m = \mathbf{v}_{\delta_2}$, so $M$ is the mass matrix of some checkerboard copula with splitting points $\Delta$ by $(i)$, thus $S$ is its value matrix.
\end{proof}

\section{Properties of the transformation $\Psi$}\label{sec:properties}

In this section we list some properties of the transformation $\Psi$, defined in equation~\eqref{eq:psi}, and of copulas in the range of $\Psi$, i.e., the set $\{\Psi(C) \mid C \text{ a copula}\}$. We first show that $\Psi$ commutes with the survival and ordinal sum operations.

\begin{lemma}\label{lem:psi_hat}
    Let $C$ be a copula. Then $\Psi(\widehat{C})=\widehat{\Psi(C)}$.
\end{lemma}

\begin{proof}
Since $\widehat{C}(t,u)=t+u-1+C(1-t,1-u)$, we have $\partial_1 \widehat{C}(t,u)=1-\partial_1 C(1-t,1-u)$ for all $t,u \in \II$. Hence, for all $u,v \in \II$
\begin{align*}
    \Psi(\widehat{C})(u,v)
    &=\int_0^1 \partial_1 \widehat{C}(t,u) \partial_1 \widehat{C}(t,v) dt \\
    &=\int_0^1 \big(1-\partial_1 C(1-t,1-u)\big)\big(1-\partial_1 C(1-t,1-v)\big) dt\\
    &=\int_0^1 \big(1-\partial_1 C(s,1-u)\big)\big(1-\partial_1 C(s,1-v)\big) ds\\
    &=1-\int_0^1 \partial_1 C(s,1-v) ds-\int_0^1 \partial_1 C(s,1-u) ds+\int_0^1 \partial_1 C(s,1-u)\partial_1 C(s,1-v) ds\\
    &=1-(1-v)-(1-u)+\Psi(C)(1-u,1-v)\\
    &=\widehat{\Psi(C)}(u,v),
\end{align*}
which finishes the proof.
\end{proof}

\begin{proposition}\label{prop:ordinal}
    If $C$ is the ordinal sum of $\{C_i\}_{i \in I}$ with respect to $\{J_i\}_{i \in I}$, then $\Psi(C)$ is the ordinal sum of $\{\Psi(C_i)\}_{i \in I}$ with respect to $\{J_i\}_{i \in I}$.
\end{proposition}

\begin{proof}
Let $J_i=(a_i,b_i)$ for all $i \in I$. By equation~\eqref{eq:ordinal} we have
\begin{equation}\label{eq:partial_C}
\partial_1C(u,v)=\begin{cases}
    \partial_1 C_i\big(\frac{u-a_i}{b_i-a_i},\frac{v-a_i}{b_i-a_i}\big), & (u,v) \in (a_i,b_i) \text{ for some } i \in I,\\
    \partial_1 M(u,v), & \text{otherwise},
\end{cases}
\end{equation}
whenever the partial derivatives on the right-hand side exist. Note also that $\partial_1 M(u,v)$ equal $1$ if $u<v$ and $0$ if $u>v$. 

Now let $u,v \in (a_i,b_i)$. Then
\begin{align*}
    \Psi(C)(u,v) &= \int_0^1 \partial_1 C(t,u) \partial_1 C(t,v) dt\\
    &= \int_{0}^{a_i} \partial_1 C(t,u) \partial_1 C(t,v) dt
    +\int_{a_i}^{b_i} \partial_1 C(t,u) \partial_1 C(t,v) dt
    +\int_{b_i}^{1} \partial_1 C(t,u) \partial_1 C(t,v) dt\\
    &=a_i+\int_{a_i}^{b_i} \partial_1 C(t,u) \partial_1 C(t,v) dt,
\end{align*}
since $\partial_1 C(t,u)=\partial_1 C(t,v)=1$ for all $t\le a_i$ and $\partial_1 C(t,u)=\partial_1 C(t,v)=0$ for all $t \ge b_i$.
Thus, by equation~\eqref{eq:partial_C}, we obtain
\begin{align*}
    \Psi(C)(u,v) &= a_i+\int_{a_i}^{b_i} \partial_1 C_i\big(\tfrac{t-a_i}{b_i-a_i},\tfrac{u-a_i}{b_i-a_i}\big) \partial_1 C_i\big(\tfrac{t-a_i}{b_i-a_i},\tfrac{v-a_i}{b_i-a_i}\big) dt\\
    &= a_i+(b_i-a_i)\int_0^1 \partial_1 C_i\big(s,\tfrac{u-a_i}{b_i-a_i}\big) \partial_1 C_i\big(s,\tfrac{v-a_i}{b_i-a_i}\big) ds\\
    &=a_i+(b_i-a_i)\Psi(C_1)\big(\tfrac{u-a_i}{b_i-a_i},\tfrac{v-a_i}{b_i-a_i}\big).
\end{align*}
Next, let $u \in (a_i,b_i)$ and $v\ge b_i$. Since $\partial_1 C(t,u)=0$ for all $t \ge b_i$ and $\partial_1 C(t,v)=1$ for all $t \le b_i$, we have
\begin{align*}
    \Psi(C)(u,v) &= \int_0^1 \partial_1 C(t,u) \partial_1 C(t,v) dt= \int_{0}^{b_i} \partial_1 C(t,u) dt=C(b_i,u)=M(b_i,u)=u=M(u,v).
\end{align*}
Now, let $u \in (a_i,b_i)$ and $v \le a_i$. Since $\partial_1 C(t,u)=1$ for all $t \le a_i$ and $\partial_1 C(t,v)=0$ for all $t \ge a_i$, we have
\begin{align*}
    \Psi(C)(u,v) &= \int_0^1 \partial_1 C(t,u) \partial_1 C(t,v) dt= \int_{0}^{a_i} \partial_1 C(t,v) dt=C(a_i,v)=v=M(u,v).
\end{align*}
By symmetry, if $v \in (a_i,b_i)$ and $u \notin (a_i,b_i)$, we get $\Psi(u,v)=M(u,v)$.
Finally, if $u,v \notin \bigcup_{i \in I} (a_i,b_i)$, then $\Psi(C)(u,v)=\Psi(M)(u,v)=M(u,v)$, since $\Psi(M)=M$.
This finishes the proof.
\end{proof}

Next, we investigate the properties of matrices associated to checkerboard copulas in the range of $\Psi$.

\begin{theorem} \label{thm:psd}
    Let $A$ be a checkerboard copula with splitting points in a symmetric mesh $\uD=\delta_2 \times \delta_2$ such that $A = \Psi(C)$ for some copula $C$.
    Then the value matrix, mass matrix, and density matrix of $A$ are all positive semidefinite. 
\end{theorem}

\begin{proof}
Let $S_A=[s_{ij}]_{ij} \in \RR^{n \times n}$ be the value matrix of $A$. By equation~\eqref{eq:psi} we have
$$s_{ij}=A(b_i,b_j)=\int_0^1 \partial_1 C(t,b_i) \partial_1 C(t,b_j) dt.$$
For any vector $\mathbf{x}=(x_1,\ldots,x_n) \in \RR^n$ we have
\begin{align*}
    \mathbf{x}^T S_A \mathbf{x} &=\sum_{i,j=1}^n s_{ij}x_ix_j=\sum_{i,j=1}^n \int_0^1 \partial_1 C(t,b_i) \partial_1 C(t,b_j) x_i x_j dt\\
    &=\int_0^1 \left(\sum_{i=1}^n \partial_1 C(t,b_i)x_i\right)^2 dt \ge 0.
\end{align*}
This proves that $S_A$ is positive semidefinite. For the mass matrix $M_A$ and the density matrix $G_A$ we have $M_A=J_n^TS_AJ_n$ and $G_A=D_{\delta_2}^{-T} M_A D_{\delta_2}^{-1}$ by Proposition~\ref{prop:connection}, so that both $M_A$ and $G_A$ are positive semidefinite as well.
\end{proof}

\begin{theorem}\label{thm:cb_cp}
    Let $A_1$ be a checkerboard copula with splitting points $\Delta=\delta_1 \times \delta_2$ and $A = \Psi(A_1)$. Then $A$  is a checkerboard copula with splitting points $\uD=\delta_2 \times \delta_2$ and its mass matrix and density matrix are completely positive.
    Furthermore,
    \begin{equation}\label{eq:G_psi}
        G_{A}=G_{A_1}^TD_{\delta_1}G_{A_1}.
    \end{equation}
\end{theorem}

\begin{proof}
By equation~\eqref{eq:psi} we have
$$A(u,v)=\int_0^1 \partial_1 A_1(t,u) \partial_1 A_1(t,v) dt.$$
Letting $(u,v) \in [b_{i-1},b_i] \times [b_{j-1},b_j]$ and differentiating with respect to $u$ ad $v$ we obtain
\begin{align*}
    \partial_1\partial_2 A(u,v) &=\int_0^1 \partial_1\partial_2 A_1(t,u) \partial_1\partial_2 A_1(t,v) dt=\sum_{k=1}^n \int_{a_{k-1}}^{a_k} \partial_1\partial_2 A_1(t,u) \partial_1\partial_2 A_1(t,v) dt.
\end{align*}
Since $A_1$ is a checkerboard copula with splitting points $\Delta$, all mixed partial derivatives on the right-hand side are constant, which justifies differentiation under the integral sign. Denoting the density matrix of $A_1$ by $G_{A_1}=[h_{ij}]_{ij} \in \RR^{m \times n}$, we obtain
\begin{align*}
    \partial_1\partial_2 A(u,v) =\sum_{k=1}^n \int_{a_{k-1}}^{a_k} h_{ki} h_{kj} dt=\sum_{k=1}^n h_{ki} h_{kj} (a_k-a_{k-1})=(G_{A_1}^TD_{\delta_1}G_{A_1})_{ij}.
\end{align*}
Since this mixed derivative is constant on $[b_{i-1},b_i] \times [b_{j-1},b_j]$ for all $i$ and $j$, we infer that $A$ is a checkerboard copula with splitting points $\uD$. Furthermore, we have shown that
\begin{equation}
    G_{A}=G_{A_1}^TD_{\delta_1}G_{A_1}=\Big(\sqrt{D_{\delta_1}}G_{A_1}\Big)^T\Big(\sqrt{D_{\delta_1}}G_{A_1}\Big).
\end{equation}
Since $\sqrt{D_{\delta_1}}G_{A_1}$ is entry-wise nonnegative matrix, we conclude that $G_A$ is completely positive. Proposition~\ref{prop:connection} implies that
$$M_A=D_{\delta_2} G_A D_{\delta_2}=\Big(\sqrt{D_{\delta_1}}G_{A_1}D_{\delta_2}\Big)^T\Big(\sqrt{D_{\delta_1}}G_{A_1} D_{\delta_2}\Big),$$
which is completely positive as well.
\end{proof}

\begin{theorem}\label{thm:cp-psi}
    Let $A$ be a checkerboard copula with splitting points in a symmetric mesh $\uD=\delta_2 \times \delta_2$, such that its density matrix is completely positive. Then $A=\Psi(A_1)$ for some checkerboard copula $A_1$ with splitting points in $\Delta=\delta_1 \times \delta_2$ for an appropriately chosen finite set $\delta_1$.
\end{theorem}

\begin{proof}
Let $|\delta_2|=n+1$. By assumption there exists a positive integer $m$ and an entry-wise nonnegative matrix $H \in \RR^{m \times n}$ such that $G_A=H^TH$.
Furthermore, assume that $m$ is minimal possible, so that $H$ has nonzero rows (any zero row can simply be omitted from $H$).
Let $\mathbf{d}=HD_{\delta_2}\mathbf{1}_n=H\mathbf{v}_{\delta_2} \in \RR^m$ and let $\delta_1=\{0,a_1,a_2,\ldots,a_m\}$, where $a_i=d_1^2+d_2^2+\ldots+d_i^2$.
Note that, by Proposition~\ref{prop:connection},
$$a_m=\mathbf{d}^T\mathbf{d}=\mathbf{1}_n^T D_{\delta_2}H^T HD_{\delta_2}\mathbf{1}_n=\mathbf{1}_n^T D_{\delta_2} G_A D_{\delta_2}\mathbf{1}_n=\mathbf{1}_n^T M_A\mathbf{1}_n=1,$$
since the sum of entries of $M_A$ is equal to $1$.
Note also that $D_{\delta_1}=D^2$, where $D=\diag(\mathbf{d})$, so that $\mathbf{v}_{\delta_1}=(d_1^2,d_2^2\ldots,d_m^2)^T$. Since the matrix $H$ has nonzero rows, vector $\mathbf{d}$ has nonzero components, so the matrix $D$ is invertible.
Furthermore, $\mathbf{v}_{\delta_1}=D_{\delta_1}\mathbf{1}_m=D^2\mathbf{1}_m$ and $\mathbf{v}_{\delta_2}=D_{\delta_2}\mathbf{1}_n$.
Now, let $G=D^{-1}H$, which is an entry-wise nonnegative matrix. Then
$$G\mathbf{v}_{\delta_2}=D^{-1}(H\mathbf{v}_{\delta_2})=D^{-1}\mathbf{d} =\mathbf{1}_m$$
and
$$G^T\mathbf{v}_{\delta_1}=(H^TD^{-1})(D^2\mathbf{1}_m)=H^T(D\mathbf{1}_m)=H^T\mathbf{d}=H^TH\mathbf{v}_{\delta_2}=G_A\mathbf{v}_{\delta_2}=\mathbf{1}_n$$
by Proposition~\ref{prop:matrices} $(ii)$.
Proposition~\ref{prop:matrices} $(ii)$ then implies that $G$ is the density matrix of some checkerboard copula $A_1$ with splitting points $\Delta=\delta_1 \times \delta_2$.
Equation~\eqref{eq:G_psi} implies that
$$G_{\Psi(A_1)}=G_{A_1}^TD_{\delta_1}G_{A_1}=G^TD_{\delta_1}G=(H^TD^{-1})D^2(D^{-1}H)=H^TH=G_A.$$
Since both $A$ and $\Psi(A_1)$ are checkerboard copulas (see Theorem~\ref{thm:cb_cp}) and the density matrix uniquely determines the checkerboard copula, we conclude that $A=\Psi(A_1)$.
\end{proof}

We collect the findings of the last three theorems in a schematic summary, along with two implications that are still open, see Figure~\ref{fig:results}.
We remark that if implication (P3) is true, then implication (P4) is false, since there exist positive semidefinite entry-wise nonegative matrices that are not completely positive, see \cite{BeShMo03}.

\begin{figure}[!ht]
    \centering
    \begin{tikzpicture}
    \def\vs{1.2}
    \draw[rounded corners=0.5cm] (-1.2-0.2*\vs,1.4*\vs) rectangle (13,-3.6*\vs);
    \path (-1.2,\vs) node[right]{Let $A$ be a checkerboard copula.};
    \path (-0.7,0) node {(P1)} (0,0) node[right, text width=5cm] {$A=\Psi(A_1)$ for some \\ checkerboard copula $A_1$}
    (5.5,0) node {$\Leftrightarrow$}
    (6.5,0) node[right ] {$M_A$, $G_A$ completely positive};
    \path (-0.7,-\vs) node {(P2)} (0,-\vs) node[right, text width=5cm] {$A=\Psi(C)$ for some \\ general copula $C$}
    (5.5,-\vs) node {$\Rightarrow$}
    (6.5,-\vs) node[right ] {$S_A$, $M_A$, $G_A$ positive semidefinite};
    \path (-0.7,-3*\vs) node {(P4)} (0,-3*\vs) node[right, text width=5cm] {$A=\Psi(C)$ for some \\ general copula $C$}
    (5.5,-3*\vs) node {$\Leftarrow$} +(0.05,0.3) node {\scriptsize{?}}
    (6.5,-3*\vs) node[right ] {$S_A$, $M_A$, $G_A$ positive semidefinite};
    \path (-0.7,-2*\vs) node {(P3)} (0,-2*\vs) node[right, text width=5cm] {$A=\Psi(C)$ for some \\ general copula $C$}
    (5.5,-2*\vs) node {$\Rightarrow$} +(-0.05,0.3) node {\scriptsize{?}}
    (6.5,-2*\vs) node[right ] {$M_A$, $G_A$ completely positive};
    \draw (-0.8-0.2*\vs,0.7)--(12.6,0.7);
\end{tikzpicture}
    \caption{Summary of the main results of Section~\ref{sec:properties} and two open implications.}
    \label{fig:results}
\end{figure}

Finally, we show that a symmetric checkerboard approximation of a copula in the range of $\Psi$ is again in the range of $\Psi$.

\begin{proposition}\label{prop:3.6}
    Let $C$ be a copula such that $C = \Psi(C_1)$ for some copula $C_1$. Let $\uD$ be a symmetric mesh and $A$ a checkerboard extension of $C|_{\uD}$. Then $A = \Psi(D_1)$ for some copula $D_1$.
\end{proposition}

\begin{proof}
Suppose $\uD=\delta_2 \times \delta_2$. Let $D_1$ be an extension of $C_1|_{\II \times \delta_2}$ to $\II^2$ that is linear in the second coordinate on every strip $\II \times [b_{j-1},b_j]$, i.e., for $(t,v) \in \II \times [b_{j-1},b_j]$ let
$$D_1(t,v)=\big(1-L_j(v)\big)C_1(t,b_{j-1})+L_j(v)C_1(t,b_j),$$
where
$$L_j(v)=\frac{v-b_{j-1}}{b_j-b_{j-1}}.$$
Then $D_1$ is a copula by standard arguments (in fact, $D_1$ is the classical extension of a subcopula $C_1|_{\II \times \delta_2}$ to a copula, used in the proof of Sklar's theorem, see \cite[Lemma~2.3.5]{Nel06}).
For all $(t,v) \in \II \times [b_{j-1},b_j]$ we have
$$\partial_1 D_1(t,v)=(1-L_j(v))\partial_1 C_1(t,b_{j-1})+L_j(v)\partial_1 C_1(t,b_j),$$
hence, for all $(u,v) \in [b_{i-1},b_i] \times [b_{j-1},b_j]$ we get
\begin{align*}
    \Psi(D_1)(u,v) &=\int_0^1 \partial_1 D_1(t,u) \partial_1 D_1(t,v) dt\\
    &= \int_0^1 \Big(\big(1-L_i(u)\big)\partial_1 C_1(t,b_{i-1})+L_i(u)\partial_1 C_1(t,b_i)\Big) \\
    & \hspace{3cm}\cdot \Big(\big(1-L_j(v)\big)\partial_1 C_1(t,b_{j-1})+L_j(v)\partial_1 C_1(t,b_j)\Big) dt\\
    &= \big(1-L_i(u)\big)\big(1-L_j(v)\big) \int_0^1 \partial_1 C_1(t,b_{i-1})\partial_1 C_1(t,b_{j-1}) dt\\
    &\hspace{1cm} +\big(1-L_i(u)\big)L_j(v) \int_0^1 \partial_1 C_1(t,b_{i-1})\partial_1 C_1(t,b_j) dt\\
    &\hspace{1cm} +L_i(u)\big(1-L_j(v)\big) \int_0^1 \partial_1 C_1(t,b_i)\partial_1 C_1(t,b_{j-1}) dt\\
    &\hspace{1cm} +L_i(u)L_j(v) \int_0^1 \partial_1 C_1(t,b_i)\partial_1 C_1(t,b_j) dt\\
    &=\big(1-L_i(u)\big)\big(1-L_j(v)\big) C(b_{i-1},b_{j-1})+\big(1-L_i(u)\big)L_j(v) C(b_{i-1},b_j)\\
    &\hspace{1cm} +L_i(u)\big(1-L_j(v)\big) C(b_i,b_{j-1}) +L_i(u)L_j(v) C(b_i,b_j).
\end{align*}
The last expression is exactly the formula for the checkerboard extension of $C|_{\uD}$, given in equation~\eqref{eq:CB}, hence, $\Psi(D_1)=A$.
\end{proof}

\section{Lower bound for the diagonal section of a copula in the range of $\Psi$}\label{sec:lower}

In this section we determine the exact lower bound for the diagonal section of a copula $C$ in the range of transformation $\Psi$ under the condition that its value at a single point is known, namely $C(a,a)=d$. We first derive the bound, which is a piecewise function, consisting of four pieces.
As we show later, for points $u$ close to $0$ or $1$ the bound is equal to the global bound for the diagonal of a copula in the range of $\Psi$ (without any constraints on its values), i.e., $C(u,u) \ge u^2$.
In the next proposition we consider the bound for points $u \ge a$, that are close to $a$.

\begin{proposition}\label{prop:spodnja1}
    Let $C$ be a copula such that $C = \Psi(C_1)$ for some copula $C_1$, and let $a \in (0,1)$. Let $d =C(a,a) \in (a^2, a]$ be the value of $C$ at a diagonal point and $a^* = \tfrac{d}{a} \in (a,1]$. Then for any point $u \in [a, a^*]$  we have 
    $$ C(u,u) \ge \frac{d}{d-a^2}(u-a)^2 + d = \frac{a^*}{a^*-a}(u-a)^2 + d.$$
\end{proposition}

\begin{proof}
Let $\uD=\{0, a, u, 1\} \times \{0, a, u, 1\}$ be a symmetric mesh and $A$ a checkerboard extension of $C|_{\uD}$. Then $A = \Psi(D_1)$ for some copula $D_1$ by Proposition \ref{prop:3.6}, so its value matrix $S_A = S_C$ is positive semidefinite by Theorem~\ref{thm:psd}.
We denote by $x = C(a,u) = C(u,a)$ and $y = C(u,u)$.
In this notation we have 
$$ S_C = \begin{bmatrix}
    d & x & a \\
    x & y  & u \\
    a & u & 1 
\end{bmatrix},$$
so 
\begin{equation}\label{eq:detSA}
   \det S_C = dy + 2axu - x^2 - a^2y - du^2 \ge 0. 
\end{equation}
Since $x = C(a,u) \ge C(a,a) = d$ and $d = aa^* \ge au$, we have
$$x+d-2au \ge 2d -2au \ge 0,$$
which implies
\begin{equation}\label{eq:audx}
    (x+d-2au)(x-d) = x^2-2axu +2aud - d^2 \ge 0.
\end{equation}
Adding inequalities \eqref{eq:detSA} and \eqref{eq:audx} we obtain
$$dy - a^2y - du^2 + 2aud - d^2 \ge 0,$$
which is equivalent to 
$$ (y-d)(d-a^2) \ge d(u-a)^2. $$
Since $d > a^2$, this implies
$$ y = C(u,u) \ge \frac{d}{d-a^2}(u-a)^2 + d.$$
which finishes the proof.
\end{proof}

Next, we consider the bound for points $u \le a$, that are close to $a$. We use a reflection argument to reduce this case to the previous one.

\begin{proposition}\label{prop:spodnja2}
    Let $C$ be a copula such that $C = \Psi(C_1)$ for some copula $C_1$ and $a \in (0,1)$. Let $d =C(a,a) \in (a^2, a]$ be the value of $C$ at a diagonal point and $a^{**} = \tfrac{a-d}{1-a} \in [0,a)$. Then for any point $u \in [a^{**}, a]$  we have 
    $$ C(u,u) \ge \frac{d+1-2a}{d-a^2}(u-a)^2 + 2(u-a) + d =\frac{1-a^{**}}{a-a^{**}}(u-a)^2 + 2(u-a) + d.$$
\end{proposition}

\begin{proof}
By Lemma~\ref{lem:psi_hat} we have $\widehat{C} = \Psi(\widehat{C_1})$. Denote $\hat{a} = 1-a$ and $\hat{u} = 1-u$. Furthermore, let 
$$\hat{d} = \widehat{C}(\hat{a},\hat{a}) = \widehat{C}(1-a,1-a) = 1-2a+C(a,a) = 1-2a+d \in (\hat{a}^2, \hat{a}],$$ 
and $a^* = \tfrac{\hat{d}}{\hat{a}}$. We have $\hat{u} \ge \hat{a}$ and $\hat{u} \le 1- a^{**} = \frac{1-2a+d}{1-a} = a^*$, so, by Proposition~\ref{prop:spodnja1},
$$ \widehat{C}(\hat{u},\hat{u}) \ge \frac{\hat{d}}{\hat{d}-\hat{a}^2}(\hat{u}-\hat{a})^2 + \hat{d} = \frac{d+1-2a}{d-a^2}(u-a)^2 + 1-2a+d.$$
It follows that 
\begin{align*}
    C(u,u) &= 2u-1 + \widehat{C}(1-u,1-u) = 2u-1 + \widehat{C}(\hat{u},\hat{u}) \\
    &\ge 2u-1 + \frac{d+1-2a}{d-a^2}(u-a)^2 + 1-2a+d \\
    &= \frac{d+1-2a}{d-a^2}(u-a)^2 + 2(u-a) + d,
\end{align*}
which finishes the proof.
\end{proof}

Our next two propositions demonstrate that the bounds derived in Propositions~\ref{prop:spodnja1} and \ref{prop:spodnja2} are sharp. In Proposition~\ref{prop:real_spodnja1} we construct a copula in the range of $\Psi$ that attains the bound on the interval $[a,1]$ and in Proposition~\ref{prop:real_spodnja2} we find a copula in the range of $\Psi$ that attains the bound on the interval $[0,a]$.

\begin{proposition} \label{prop:real_spodnja1}
    Let $a \in (0,1)$, $d \in (a^2, a]$ and $a^* = \tfrac{d}{a} \in (a,1]$. If $d<a$ then $a^* < 1$ and there exists a checkerboard copula $A$ with splitting points $\uD = \delta_2 \times \delta_2$, where $\delta_2 = \{0, a, a^*, 1\}$, such that $A = \Psi(A_1)$ for some checkerboard copula $A_1$ and its diagonal section is
    $$A(u,u) = \begin{cases}
			\frac{d}{a^2}u^2; & \text{if }0 \le u < a, \\
			\frac{d}{d-a^2}(u-a)^2 + d; & \text{if }a \le u < a^*, \\
			u^2; & \text{if }a^* \le u \le 1. 
    \end{cases}$$
    If $d=a$ then there exists a checkerboard copula $A$ with splitting points $\uD = \delta_2 \times \delta_2$, where $\delta_2 = \{0, a, 1\} $, such that $A = \Psi(A_1)$ for some checkerboard copula $A_1$ and its diagonal section is
    $$A(u,u) = \begin{cases}
			\frac{1}{a}u^2; & \text{if }0 \le u < a, \\
			\frac{1}{1-a}(u-a)^2 + a; & \text{if }a \le u \le 1. 
    \end{cases}$$
\end{proposition}

\begin{proof}
Suppose first that $d < a$. Let 
$$S = \begin{bmatrix}
    d & d & a \\
    d & (a^*)^2  & a^* \\
    a & a^* & 1 
\end{bmatrix}.$$
We have 
$$
\begin{aligned}
M=J_3^T S J_3 &= \begin{bmatrix}
    d & 0 & a-d \\
    0 & (a^*)^2-d  & a^*+d-a-(a^*)^2 \\
    a-d & a^*+d-a-(a^*)^2 & (1-a^*)^2 
\end{bmatrix} \\
&= \begin{bmatrix}
    d & 0 & a(1-a^*) \\
    0 & \tfrac{d(a^*-a)^2}{d-a^2}  & (a^*-a)(1-a^*) \\
    a(1-a^*) & (a^*-a)(1-a^*) & (1-a^*)^2 
\end{bmatrix}.
\end{aligned}$$
The matrix $M$ is entry-wise nonnegative, 
so matrix $S$ is
the value matrix of some checkerboard copula $A$ with splitting points $\uD$ by Proposition \ref{prop:matrices} $(iii)$. 
The matrix $S$ is positive semidefinite, since its principal $2\times 2$ minors are 
$$d((a^*)^2-d) > 0, \quad (a^*)^2 - (a^*)^2 = 0, \quad d-a^2 > 0,$$
and
$$\det S = d(a^*)^2 +2ada^* - a^2(a^*)^2 - d^2 - d(a^*)^2 = 0, $$
because of $a^*=\frac{d}{a}$. So also the density matrix $G_A$ of the copula $A$ 
$$G_A = \begin{bmatrix}
    \frac{1}{a} & 0 & 0 \\
    0 & \frac{1}{a^*-a}  & 0 \\
    0 & 0 & \frac{1}{1-a^*} 
\end{bmatrix}M\begin{bmatrix}
    \frac{1}{a} & 0 & 0 \\
    0 & \frac{1}{a^*-a}  & 0 \\
    0 & 0 & \frac{1}{1-a^*} 
\end{bmatrix}=\begin{bmatrix}
    \frac{d}{a^2} & 0 & 1 \\
    0 & \frac{d}{d-a^2}  & 1 \\
    1 & 1 & 1 
\end{bmatrix}$$
is positive semidefinite. Since every $3\times 3$ positive semidefinite matrix is completely positive (see Section~\ref{sec:prelim}), this implies that copula $A$ is of the form $A = \Psi(A_1)$ for some checkerboard copula $A_1$ by Theorem \ref{thm:cp-psi}. If $0 \le u < a$ then $A(u,u) = g_{11}u^2 = \frac{d}{a^2}u^2$. If $a \le u < a^*$, we have $A(a,u) \in [A(a,a), A(a,a^*)] = \{d\}$, so 
$$A(u,u) = V_A([a,u]^2) - A(a,a) + 2A(a,u) = g_{22}(a-u)^2 - d + 2d = \frac{d}{d-a^2}(u-a)^2 + d.$$
Finally, if $a^* \le u \le 1$, then 
$$A(u,u) = V_A([u,1]^2) - 1 + 2u = g_{33}(1-u)^2 - 1 + 2u = u^2,$$
so the case $d<a$ is finished.
If $d=a$ let 
$$S = \begin{bmatrix}
    a & a \\
    a & 1 
\end{bmatrix}.$$
Then $S$ is a positive semidefinite matrix with $J_2^T S J_2$ entry-wise nonnegative, so it is a value matrix of a checkerboard copula $A$ with splitting points $\uD = \delta_2 \times \delta_2$, where $\delta_2 = \{0, a, 1\} $, with $A = \Psi(A_1)$ for some checkerboard copula $A_1$, and with diagonal section 
    $$A(u,u) = \begin{cases}
			\frac{1}{a}u^2; & \text{if }0 \le u < a, \\
			\frac{1}{1-a}(u-a)^2 + a; & \text{if }a \le u \le 1, 
    \end{cases}$$
    which completes the proof.
\end{proof}

\begin{proposition} \label{prop:real_spodnja2}
    Let $a \in (0,1)$, $d \in (a^2, a]$ and $a^{**} = \tfrac{a-d}{1-a} \in [0,a)$. If $d < a$ then $a^{**} > 0$  and there exists a checkerboard copula $A$ with splitting points $\uD = \delta_2 \times \delta_2$,  where $\delta_2 = \{0, a^{**},  a, 1\} $, such that $A = \Psi(A_1)$ for some checkerboard copula $A_1$ and its diagonal section is
    $$A(u,u) = \begin{cases}
			u^2; & \text{if }0 \le u < a^{**}, \\
			\frac{d+1-2a}{d-a^2}(u-a)^2 + 2(u-a) + d; & \text{if }a^{**} \le u < a, \\
			\frac{d+1-2a}{(1-a)^2}(u-a)^2 + \frac{2(a-d)}{1-a}(u-a)+d; & \text{if }a \le u \le 1. 
    \end{cases}$$
    If $d=a$ then there exists a checkerboard copula $A$ with splitting points $\uD = \delta_2 \times \delta_2$ where $\delta_2 = \{0, a, 1\} $ such that $A = \Psi(A_1)$ and its diagonal section is
    $$A(u,u) = \begin{cases}
			\frac{1}{a}(u-a)^2 + 2u-a; & \text{if }0 \le u < a, \\
			\frac{1}{1-a}(u-a)^2 + a; & \text{if }a \le u \le 1. 
   \end{cases}$$
\end{proposition}

\begin{proof}
Suppose that $d < a$ and let $\hat{a} = 1-a$, $\hat{d} = 1-2a+d$, and $a^* = \tfrac{\hat{d}}{\hat{a}}$. Then $\hat{d} \in (\hat{a}^2, \hat{a})$, so, by Proposition~\ref{prop:real_spodnja1}, there exists a checkerboard copula $B$ with splitting points $\widehat{\uD} = \delta_2' \times \delta_2'$, where $\delta_2' = \{0, \hat{a}, a^*, 1\} $, such that $B = \Psi(B_1)$ for some checkerboard copula $B_1$ and its diagonal section is
    $$B(u,u) = \begin{cases}
			\frac{\hat{d}}{\hat{a}^2}u^2; & \text{if }0 \le u < \hat{a}, \\
			\frac{\hat{d}}{\hat{d}-\hat{a}^2}(u-\hat{a})^2 + \hat{d}; & \text{if }\hat{a} \le u < a^*, \\
			u^2; & \text{if }a^* \le u \le 1. 
    \end{cases}$$
By Lemma~\ref{lem:checkerboard_hat}, it follows that its survival copula $A = \widehat{B}$ is a checkerboard copula with splitting points $\uD = \delta_2 \times \delta_2$, where $\delta_2 = \{0, 1-a^*, a, 1\} $, such that $A = \Psi(A_1)$ for some checkerboard copula $A_1$. Note that $1-a^* = \tfrac{a-d}{1-a} = a^{**}$. The diagonal section of $A$ is 
\begin{align*}
    A(u,u) &= 2u-1 + B(1-u,1-u) \\
    &= 2u-1 + \begin{cases}
			\frac{\hat{d}}{\hat{a}^2}(1-u)^2; & \text{if }0 \le 1-u < \hat{a}, \\
			\frac{\hat{d}}{\hat{d}-\hat{a}^2}(1-u-\hat{a})^2 + \hat{d}; & \text{if }\hat{a} \le 1-u < a^*, \\
			(1-u)^2; & \text{if }a^* \le 1-u \le 1 
    \end{cases} \\
    &= \begin{cases}
			u^2; & \text{if }0 \le u < a^{**}, \\
			\frac{d+1-2a}{d-a^2}(u-a)^2 + 2(u-a) + d; & \text{if }a^{**} \le u < a, \\
			\frac{d+1-2a}{(1-a)^2}(u-a)^2 + \frac{2(a-d)}{1-a}(u-a)+d; & \text{if }a \le u \le 1.
    \end{cases}
\end{align*}
If $d=a$, similar arguments as above show the existence of a checkerboard copula $A$ with splitting points $\uD = \delta_2 \times \delta_2$, where $\delta_2 = \{0, a, 1\} $, with $A = \Psi(A_1)$ for some checkerboard copula $A_1$, and the diagonal section simplifies to
    $$A(u,u) = \begin{cases}
			\frac{1}{a}(u-a)^2 + 2u-a; & \text{if }0 \le u < a, \\
			\frac{1}{1-a}(u-a)^2 + a; & \text{if }a \le u \le 1, 
   \end{cases}$$
so the proof is complete.
\end{proof}

We collect our findings in our main theorem of this section.

\begin{theorem}\label{thm:main_lower}
    Let $C$ be a copula such that $C = \Psi(C_1)$ for some copula $C_1$ and $a \in (0,1)$. If $A(a,a) = a^2$, then $A(u,u) \ge u^2$ for any $u \in \II$. If $d =C(a,a) \in (a^2, a]$ is the value of $C$ at a diagonal point, let $a^* = \tfrac{d}{a} \in (a,1]$ and $a^{**} = \tfrac{a-d}{1-a} \in [0,a)$. Then the lower bound for the diagonal section of $C$ is 
    $$C(u,u) \ge \begin{cases}
			\frac{d+1-2a}{d-a^2}(u-a)^2 + 2(u-a) + d; & \text{if }a^{**} \le u \le a, \\
			\frac{d}{d-a^2}(u-a)^2 + d; & \text{if }a \le u \le a^*, \\
			u^2; & \text{otherwise}. 
    \end{cases}$$
    Every point on the lower bound is attained by some checkerboard copula $A$ of the form $A = \Psi(A_1)$ for some checkerboard copula $A_1$.
\end{theorem}

\begin{proof}
Note that, by \cite[Corollary~1]{Fu24}, the bound $C(u,u) \ge u^2$ is valid for any copula of the form $C = \Psi(C_1)$ for some copula $C_1$. If $A(a,a) = a^2$, then the bound $C(u,u)=u^2$ is attained by the product copula $\Pi(u,v) = uv$, which satisfies $\Pi = \Psi(\Pi)$.

If $d =C(a,a) \in (a^2, a]$ then Propositions \ref{prop:spodnja1} and \ref{prop:spodnja2} prove the bound for $u \in [a, a^*]$ and $u \in [a^{**},a]$.
If $u \notin [a^{**},a^*]$, the bound is valid for any copula of the form $C = \Psi(C_1)$. Proposition \ref{prop:real_spodnja1} demonstrates that the bound is attained by some checkerboard copula $A$ of the form $A = \Psi(A_1)$ for some checkerboard copula $A_1$ for any $u \in [a,1]$, while Proposition \ref{prop:real_spodnja2} shows that it is attained in the same way for any $u \in [0, a]$. 
\end{proof}

It is interesting to note that the bound in Theorem~\ref{thm:main_lower} has a nice geometric interpretation as explained in the following remark.
\begin{remark}
The function 
$$r(u) = \frac{d}{d-a^2}(u-a)^2 + d$$
satisfies $r(a) = d$, $r'(a) = 0$, $r(a^*) = (a^*)^2$ and $r'(a^*) = 2a^*$. It is the unique quadratic function with the properties $r(a) = d$ and $r'(a) = 0$ which matches with $a^2$ and its derivative at some point in the interval $[a,1]$. Similarly,  the function 
$$s(u) = \frac{d+1-2a}{d-a^2}(u-a)^2 + 2(u-a) + d$$
satisfies $s(a) = d$, $s'(a) = 2$, $s(a^{**}) = (a^{**})^2$ and $s'(a^{**}) = 2a^{**}$. It is the unique quadratic function with the properties $s(a) = d$ and $s'(a) = 2$ which matches with $a^2$ and its derivative at some point in the interval $[0,a]$.
\end{remark}

\begin{remark}
Note that the lower bound of the diagonal section of any copula $C = \Psi(C_1)$ with $C(a,a) \in (a^2, a]$ is attained for all points in the interval $[a,1]$ simultaneously, i.e., with the same copula. Similarly, it is attained for all points in the interval $[0,a]$ simultaneously. But there does not exist a copula 
$C = \Psi(C_1)$ with diagonal section equal to the bound on the whole interval $\II$, unless $C(a,a)=a^2$ or $C(a,a)=a$.
\end{remark}

We conclude the section with an example, demonstrating our result.

\begin{example}\label{ex:lower}
 Let $a = \tfrac25$ and $d=\tfrac{8}{25}\in[\tfrac{4}{25},\tfrac25]$. Then $a^* = \tfrac{d}{a} = \tfrac45$ and $a^{**} = \tfrac{a-d}{1-a} = \tfrac{2}{15}$. The lower bound for the diagonal section of any copula $C$ of the form $C = \Psi(C_1)$ for some copula $C_1$, satisfying the condition $C(a,a) = d$ is
    $$C(u,u) \ge \begin{cases}
			\tfrac{13}{4} u^2 - \frac35 u + \tfrac{1}{25}; & \text{if }\tfrac{2}{15} \le u \le \tfrac25, \\
			2 u^2 - \frac85 u + \tfrac{16}{25}; & \text{if }\tfrac25 \le u \le \tfrac45, \\
			u^2; & \text{otherwise}. 
    \end{cases}$$
The checkerboard copula $A$ of the form $A = \Psi(A_1)$ for some copula $A_1$ that attains the lower bound for all $u \in [\tfrac25,1]$ has splitting points $\uD = \delta_2 \times \delta_2$, where $\delta_2 = \{0, \tfrac25, \tfrac45, 1\}$ and the density matrix
$$G_A = \begin{bmatrix}
    2 & 0 & 1 \\
    0 & 2  & 1 \\
    1 & 1 & 1 
\end{bmatrix},$$
while the checkerboard copula $B$ of the form $B = \Psi(B_1)$ for some copula $B_1$ that attains the lower bound for all $u \in [0,\tfrac25]$ has splitting points $\uD' = \delta'_2 \times \delta'_2$, where $\delta'_2 = \{0, \tfrac{2}{15}, \tfrac25, 1\}$ and the density matrix
$$G_B = \begin{bmatrix}
    1 & 1 & 1 \\
    1 & \tfrac{13}{4}  & 0 \\
    1 & 0 & \tfrac{13}{9} 
\end{bmatrix}.$$
Figure \ref{fig:lower} shows the graph of the lower bound together with diagonal sections of copulas $M$ and $\Pi$, diagonal sections of copulas $A$ and $B$, and their densities as heatmaps, where darker shades represent higher densities. 
\end{example}

\begin{figure}[ht]
\centering
\def\a{2/5}
\def\d{8/25}
\def\aa{4/5}
\def\aaa{2/15}
\begin{tikzpicture}[scale=4]
\draw[->] (-0.05,0) -- (1.1,0) node[right] {$u$};
\draw[->] (0,-0.05) -- (0,1.1);
\draw (\a,0.01) -- (\a,-0.01) node[below]{$a$};
\draw (\aa,0.01) -- (\aa,-0.01) node[below]{$a^*$};
\draw (\aaa,0.01) -- (\aaa,-0.01) node[below]{$a^{**}$};
\draw (1,0.01) -- (1,-0.01) node[below]{$1$};
\draw (0.01,\d) -- (-0.01,\d) node[left]{$d$};
\draw (0.01,\aa*\aa) -- (-0.01,\aa*\aa) node[left]{$(a^*)^2$};
\draw (0.01,\aaa*\aaa) -- (-0.01,\aaa*\aaa) node[left]{$(a^{**})^2$};
\draw (0.01,1) -- (-0.01,1) node[left]{$1$};
\draw[-] (0,0) -- (1,1);
\draw[domain=0:1,samples=100] plot (\x,\x*\x);
\draw[black, very thick,domain=0:\aaa,samples=100] plot (\x,\x*\x);
\draw[black, very thick,domain=\aa:1,samples=100] plot (\x,\x*\x);
\draw[black, very thick,domain=\aaa:\a,samples=100] plot (\x,13/4*\x*\x  - 3/5*\x + 1/25);
\draw[black, very thick,domain=\a:\aa,samples=100] plot (\x,2*\x*\x -8/5*\x + 16/25);
\fill (\a,\d) circle (0.01) (\aa,\aa*\aa) circle (0.01) (\aaa,\aaa*\aaa) circle (0.01);
\end{tikzpicture} 
\hspace{1cm}
\begin{tikzpicture}[scale=4]
\draw[->] (-0.05,0) -- (1.1,0) node[right] {$u$};
\draw[->] (0,-0.05) -- (0,1.1);
\draw (\a,0.01) -- (\a,-0.01) node[below]{$a$};
\draw (\aa,0.01) -- (\aa,-0.01) node[below]{$a^*$};
\draw (\aaa,0.01) -- (\aaa,-0.01) node[below]{$a^{**}$};
\draw (1,0.01) -- (1,-0.01) node[below]{$1$};
\draw (0.01,\d) -- (-0.01,\d) node[left]{$d$};
\draw (0.01,\aa*\aa) -- (-0.01,\aa*\aa) node[left]{$(a^*)^2$};
\draw (0.01,\aaa*\aaa) -- (-0.01,\aaa*\aaa) node[left]{$(a^{**})^2$};
\draw (0.01,1) -- (-0.01,1) node[left]{$1$};
\draw[-] (0,0) -- (1,1);
\draw[domain=0:1,samples=100] plot (\x,\x*\x);
\draw[blue, very thick,domain=0:\a,samples=100] plot (\x,2*\x*\x);
\draw[blue, very thick,domain=\aa:1,samples=100] plot (\x,\x*\x);
\draw[blue, very thick,domain=\a:\aa,samples=100] plot (\x,2*\x*\x -8/5*\x + 16/25);
\fill[blue] (\a,\d) circle (0.01);
\fill[blue] (\aa,\aa*\aa) circle (0.01);
\draw[red, very thick,domain=\a:1,samples=100] plot (\x,13/9*\x*\x - 8/9*\x + 4/9);
\draw[red, very thick,domain=\aaa:\a,samples=100] plot (\x,13/4*\x*\x  - 3/5*\x + 1/25);
\draw[red, very thick,domain=0:\aaa,samples=100] plot (\x,\x*\x);
\fill[red] (\a,\d) circle (0.01);
\fill[red] (\aaa,\aaa*\aaa) circle (0.01);
\end{tikzpicture}
\\[5mm]
\begin{tikzpicture}[scale=4]
\filldraw[black!32] (0,0) -- (0,\a) -- (\a,\a) -- (\a,0) -- cycle;
\filldraw[black!32] (\a,\a) -- (\aa,\a) -- (\aa,\aa) -- (\a,\aa) -- cycle;
\filldraw[black!16] (0,\aa) -- (0,1) -- (1,1) -- (1,\aa) -- cycle;
\filldraw[black!16] (\aa,0) -- (1,0) -- (1,1) -- (\aa,1) -- cycle;
\draw[thick] (0,0) -- (1,0) -- (1,1) -- (0,1) -- cycle;
\draw (\a,1) -- (\a,0) node[below]{$a$};
\draw (\aa,1) -- (\aa,0) node[below]{$a^*$};
\draw (1,1) -- (1,0) node[below]{$1$};
\draw (1,\a) -- (0,\a) node[left]{$a$};
\draw (1,\aa) -- (0,\aa) node[left]{$a^*$};
\draw (1,1) -- (0,1) node[left]{$1$};
\end{tikzpicture}
\hspace{2cm}
\begin{tikzpicture}[scale=4]
\filldraw[black!23] (1,1) -- (1,\a) -- (\a,\a) -- (\a,1) -- cycle;
\filldraw[black!52] (\aaa,\aaa) -- (\a,\aaa) -- (\a,\a) -- (\aaa,\a) -- cycle;
\filldraw[black!16] (0,\aaa) -- (1,\aaa) -- (1,0) -- (0,0) -- cycle;
\filldraw[black!16] (\aaa,0) -- (\aaa,1) -- (0,1) -- (0,0) -- cycle;
\draw[thick] (0,0) -- (1,0) -- (1,1) -- (0,1) -- cycle;
\draw (\a,1) -- (\a,0) node[below]{$a$};
\draw (\aaa,1) -- (\aaa,0) node[below]{$a^{**}$};
\draw (1,1) -- (1,0) node[below]{$1$};
\draw (1,\a) -- (0,\a) node[left]{$a$};
\draw (1,\aaa) -- (0,\aaa) node[left]{$a^{**}$};
\draw (1,1) -- (0,1) node[left]{$1$};
\end{tikzpicture}
    \caption{The graph of the lower bound together with diagonal sections of copulas $M$ and $\Pi$ (top left), diagonal sections of copulas $A$ in blue and $B$ in red (top right) from Example \ref{ex:lower}, and densities of copulas $A$ (bottom left) and $B$ (bottom right) as heatmaps.}
    \label{fig:lower}
\end{figure}
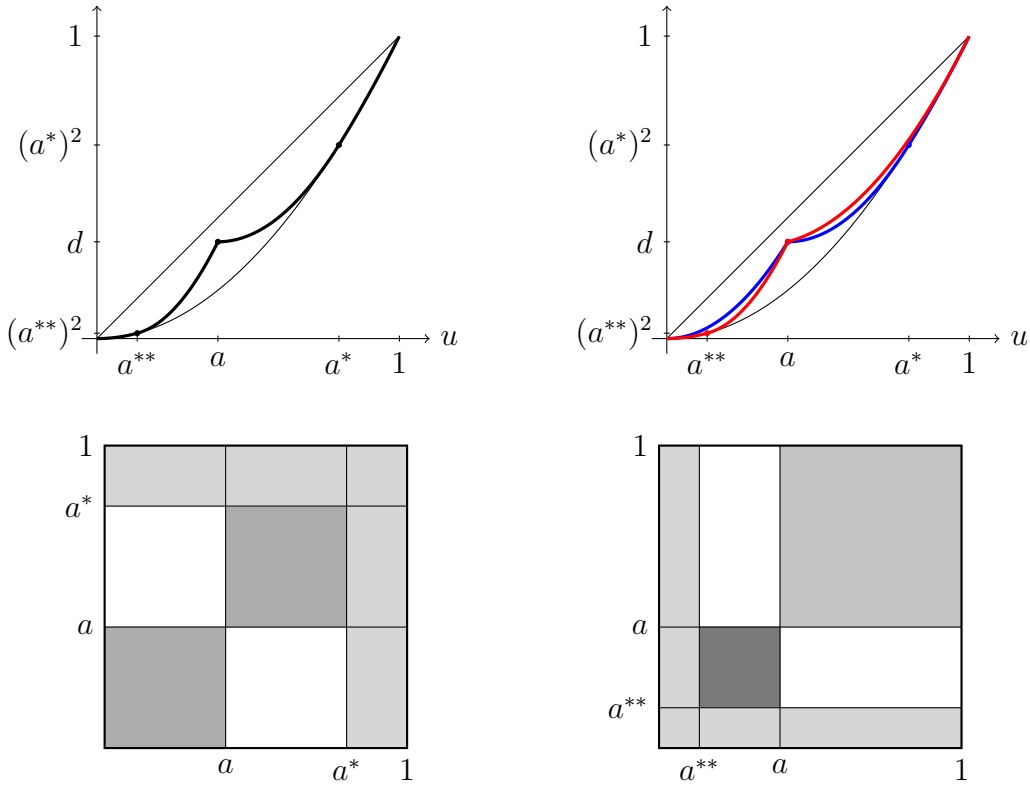

\section{Upper bound for the diagonal section of a copula in the range of $\Psi$}\label{sec:upper}

In this section we determine the exact upper bound for the diagonal section of a copula $C$ in the range of transformation $\Psi$ under the condition that its value at a single point is known.
In the previous section the point, where the value of the diagonal section of $C$ was known, was denoted by $a$, while the function variable of the bound was denoted by $u$.
In this this section we reverse the notation -- we consider the constraint $C(u,u)=w$, where $u$ is a given point, and denote the function variable of the bound by $a$. This is done intentionally, since in the crucial step of the proof of the bound (Proposition~\ref{prop:zgornja1}) we will reverse the roles of the points $a$ and $u$ and use Proposition~\ref{prop:spodnja1}.

We first introduce a function $f$ which will later give part of the upper bound, and derive some of its properties that we will need in the proofs.

\begin{lemma} \label{function_f}
Let $u \in (0,1)$, $w \in [u^2, u]$ and $u^* = \tfrac{w-u^2}{w+1-2u}$. Let a function $f \colon [0, u] \to [0,\infty)$ be defined by
\begin{align}
    f(a)=f_{u,w}(a) &= a^2+\tfrac14 \left(\sqrt{w- (u-2a)^2} + \sqrt{w- u^2}\right)^2 \label{eq:f} \\
    &= au + \tfrac12(w-u^2) +\sqrt{(w-u^2)(a(u-a)+\tfrac14(w-u^2))} \nonumber \\
    &= au + \tfrac12(w-u^2) + \tfrac12\sqrt{(w-u^2)(w- (u-2a)^2)}. \nonumber 
\end{align} 
Then $u^* \in [0,u]$,  $f(0) = w- u^2$, $f(u) = w$, $f(u^*) = u^*$, $f$ is twice differentiable, increasing, concave, $f^\prime(a) \in [0,1)$ for every $a \in [u^*, u]$, and $u-a-w+f(a) \ge 0$ for every $a \in [u^*, u]$.
\end{lemma}

\begin{proof}
A short calculation reveals that all three forms of function $f$ are equal. Let us use the abbreviation $\hat{x} = 1 - x$ for each $x \in [0,1]$.
Since $w \in [u^2, u]$, there is a number $t \in [0,1]$ such that $w = (1-t)u^2+tu$. It follows that
$$w-u^2 = t u (1-u) = t u \hat{u},$$
so that 
$$ u^* = \frac{w-u^2}{w-u^2+(1-u)^2} = \frac{t u \hat{u}}{t u \hat{u}+\hat{u}^2} = 
\frac{t u}{t u +\hat{u}} = \frac{t u}{t u +1-u} = \frac{t u}{1- \hat{t} u}.$$
Since $1- \hat{t} u \ge 1- \hat{t} = t$, we obtain that $u^* \in [0,u]$.
Now, we have 
\begin{align*}
    f(0) &= \tfrac14 \left(2 \sqrt{w- u^2}\right)^2= w-u^2, \text{ and} \\
    f(u) &= u^2+\tfrac14 \left(2 \sqrt{w- u^2}\right)^2 = w.
\end{align*}
Since 
$$ u- u^* = u - \frac{t u}{1- \hat{t} u} = \frac{u- \hat{t} u^2 - t u}{1- \hat{t} u} = 
\frac{\hat{t} u \hat{u}}{1-\hat{t} u} , $$
we obtain from the second form of $f$
\begin{align*}
     f(u^*) &= \frac{tu^2}{1-\hat{t} u}+ \frac12 t u \hat{u} + \sqrt{tu\hat{u} \left( \frac{t u}{1- \hat{t} u} \cdot \frac{\hat{t} u \hat{u}}{1-\hat{t} u}+\frac14 tu\hat{u}\right)} \\
     &= \frac{tu(2u+\hat{u}-\hat{t}u\hat{u})}{2(1-\hat{t} u)}+\frac{t u \hat{u}}{2(1-\hat{t} u)} \sqrt{4 \hat{t} u + (1-\hat{t}u)^2} \\
     &= \frac{tu}{2(1-\hat{t}u)}\left( 2u+\hat{u}-\hat{t}u\hat{u} + \hat{u} (1+\hat{t} u)\right) \\
     &= \frac{t u}{1- \hat{t} u} = u^*. 
\end{align*}
If $w=u^2$ then $f(a) = ua$, so $f'(a) = u \in (0,1)$, $f''(a) = 0$, and 
$$u-a-w+f(a) = u-a -u^2 +ua = (1-u)(u-a)\ge 0$$
for every $a \le u$.

From now on we suppose that $w > u^2$. We differentiate $f$ in the third form, to obtain that 
$$ f^\prime (a) = u + \sqrt{w- u^2} \, \frac{u-2a}{\sqrt{w-(u-2a)^2}} $$
and 
$$ f^{\prime \prime} (a) = - \frac{2 w \sqrt{w- u^2}}{(w-(u-2a)^2)^{3/2}} . $$
Clearly, $f^{\prime \prime} (a) < 0$ for all $a \in [0,u]$, and so the function $f$ is concave,
and the function $f^\prime$ is decreasing.
Since $f^\prime(u) = 0$, we have $f^\prime (a) \ge 0$ for all $a \in [0,u]$, so that 
the function $f$ is increasing. 

Define a function $f_0 \colon [0, u] \to \mathbb{R}$ by $f_0(a) = f(a) - a$. 
Since $f_0^{\prime \prime} (a) = f^{\prime \prime} (a) < 0$ for all $a \in [0,u]$, the function $f_0$ is also concave. 
Since $f_0(0) = f(0) = w- u^2 > 0$, $f_0(u) = f(u) - u = w- u \le 0$ and 
$f_0(u^*) = f(u^*) - u^* = 0$, the concavity of $f_0$ implies that the number $u^*$ is the only zero of $f_0$ and 
$f_0^\prime (a) < 0$ for all $a \in [u^*, u]$. It follows that 
$f^\prime(a)= f_0^\prime(a) + 1 \in [0,1)$ for all $a \in [u^*, u]$.

Finally, define a function $h \colon [u^*, u] \to \mathbb{R}$ by $h(a) = u-a-w+f(a)$.
Then $h^\prime(a) = f^\prime (a) - 1 < 0$ for all $a \in [u^*, u]$ and 
$h(u) = 0$, so that $h(a) \ge 0$ for all $a \in [u^*, u]$.
This completes the proof of the lemma.
\end{proof}

The other crucial part of the upper bound will be given by function $g$, which is obtained from function $f$ by means of a specific reflection. The definition and properties of function $g$ are given in our next lemma.

\begin{lemma} \label{function_g}
Let $u \in (0,1)$, $w \in [u^2, u]$ and $u^{**} = \tfrac{u^2}{w} \in [u,1]$. Let a function $g \colon [u, 1] \to [0,\infty)$ be defined by
\begin{equation} \label{eq:g}
    g(a) =g_{u,w}(a)= au + a - u + \tfrac12(w-u^2) +\sqrt{(w-u^2)((1-a)(a-u)+\tfrac14(w-u^2))}.
\end{equation}
Then $g(u) = w$, $g(u^{**}) = u^{**}$, $g$ is twice differentiable, increasing, $2$-Lipschitz, concave, $g'(a) \in [1,2]$ for every $a \in [u, u^{**}]$, and $u-a-w+g(a) \ge 0$ for every $a \in [u, u^{**}]$.
\end{lemma}

\begin{proof}
We apply the survival transformation to the function $f$ from Lemma \ref{function_f}.
Let $\hat{a}=1-a$, $\hat{u}=1-u$, and $\hat{w}=w+1-2u$. Then  $\hat{w}-\hat{u}^2=(w+1-2u) - (1-u)^2 = w- u^2$. 
Note that
\begin{align*}
    g_{u,w}(a) &= 2a-1+(1-a)(1-u)+ \tfrac12(w-u^2) + \sqrt{(w-u^2)((1-a)(a-u)+\tfrac14(w-u^2))} \\
    &= 2a-1+\hat{a}\hat{u}+ \tfrac12(\hat{w}-\hat{u}^2) + \sqrt{(\hat{w}-\hat{u}^2)(\hat{a}(\hat{u}-\hat{a})+\tfrac14(\hat{w}-\hat{u}^2))} \\
    &=2a-1+f_{\hat{u},\hat{w}}(\hat{a})= 2a-1+f_{1-u,w+1-2u}(1-a).
\end{align*}
Furthermore, let $u^* = \tfrac{\hat{w}-\hat{u}^2}{\hat{w}+1-2\hat{u}}$. Then 
\begin{equation}\label{eq:1-u*}
1-u^*=1-\frac{w+1-2u-(1-u)^2}{w+1-2u + 1- 2(1-u)} = 1-\frac{w-u^2}{w} = \frac{u^2}{w}=u^{**} .
\end{equation}
Now, the properties of the function $g$ follow easily from the corresponding properties 
of the function $f$ from Lemma \ref{function_f}, and so we omit the details.
\end{proof}

In next two propositions we prove that the functions $f$ and $g$, defined in \eqref{eq:f} and \eqref{eq:g}, are upper bounds for the diagonal section of a copula $C$ in the range of $\Psi$ satisfying the constraint $C(u,u)=w$.

\begin{proposition} \label{prop:zgornja1}
    Let $C$ be a copula such that $C = \Psi(C_1)$ for some copula $C_1$, and let $u \in (0,1)$. Let $w =C(u,u) \in [u^2, u]$ be the value of $C$ at a diagonal point. Then for any point $a \in [0, u]$  we have 
    $$ C(a,a) \le f(a) = a^2+\tfrac14 \left(\sqrt{w- (u-2a)^2} + \sqrt{w- u^2}\right)^2. $$
\end{proposition}

\begin{proof}
Let $a \le u$ and denote $d = C(a,a)$ and $a^* = \tfrac{d}{a}$. If $u > a^*$ then $d < au$, so 
\begin{align*}
C(a,a) = d < au &\le au + \tfrac12(w-u^2) + \tfrac12\sqrt{(w-u^2)(w- (u-2a)^2)} \\
&= f(a) = a^2+\tfrac14 \left(\sqrt{w- (u-2a)^2} + \sqrt{w- u^2}\right)^2.
\end{align*}

Suppose now that $u \le a^*$. If $d = a^2$ then $a^* = a = u$ and $C(a,a) = C(u,u) = w = f(u)= f(a)$ by Lemma \ref{function_f}.
If $d > a^2$ then we have by Proposition \ref{prop:spodnja1}
$$w = C(u,u) \ge \frac{d}{d-a^2}(u-a)^2 + d ,$$
which is equivalent to 
$$ d^2 - (2au + w - u^2)d + wa^2 \le 0.$$
This implies
\begin{align*}
C(a,a) = d &\le \tfrac12(2au + w - u^2) +  \sqrt{\tfrac14(2au + w - u^2)^2 - wa^2} \\
&= au + \tfrac12(w-u^2) +\tfrac12\sqrt{(w-u^2)(w- (u-2a)^2)} \\
&= f(a) = a^2+\tfrac14 \left(\sqrt{w- (u-2a)^2} + \sqrt{w- u^2}\right)^2
\end{align*}
by Lemma \ref{function_f}.
\end{proof}

\begin{proposition} \label{prop:zgornja2}
    Let $C$ be a copula such that $C = \Psi(C_1)$ for some copula $C_1$ and $u \in (0,1)$. Let $w =C(u,u) \in [u^2, u]$ be the value of $C$ at a diagonal point. Then for any point $a \in [u, 1]$  we have 
    $$ C(a,a) \le g(a) = au + a - u + \tfrac12(w-u^2) +\sqrt{(w-u^2)(au-a^2+a-u+\tfrac14(w-u^2))}. $$
\end{proposition}

\begin{proof}
By Lemma~\ref{lem:psi_hat} we have $\widehat{C} = \Psi(\widehat{C_1})$. Denote $\hat{a} = 1-a$ and $\hat{u} = 1-u$. Furthermore, let 
$$\hat{w} = \widehat{C}(\hat{u},\hat{u}) = 1-2u+w \in [\hat{u}^2, \hat{u}].$$ 
We have $\hat{a} \in [0, \hat{u}]$, so by Proposition~\ref{prop:zgornja1}
$$ \widehat{C}(\hat{a},\hat{a}) \le \hat{a}^2+\tfrac14 \left(\sqrt{\hat{w}- (\hat{u}-2\hat{a})^2} + \sqrt{\hat{w}- \hat{u}^2}\right)^2=f_{\hat{u},\hat{w}}(\hat{a}).$$
It follows that 
\begin{align*}
    C(a,a) = 2a-1 + \widehat{C}(\hat{a},\hat{a}) \le 2a-1 + f_{\hat{u},\hat{w}}(\hat{a})= g(a)
\end{align*}
by the proof of Lemma \ref{function_g}.
\end{proof}

Next, we prove that the upper bound given by function $f$ is sharp on the interval $[u^*,u]$. For any point in $a \in [u^*,u)$ we explicitly construct a copula $A$ in the range of $\Psi$ such that $A(a,a)=f(a)$.

\begin{proposition} \label{prop:real_zgornja1}
    Let $u \in (0,1)$, $w \in [u^2, u]$,  $u^* = \tfrac{w-u^2}{w+1-2u} \in [0,u]$ and $a \in [u^*, u)$. Then there exists a checkerboard copula $A$ with splitting points $\uD = \delta_2 \times \delta_2$, with $\delta_2 = \{0, a, u, 1\} $, such that $A = \Psi(A_1)$ for some checkerboard copula $A_1$, $A(u,u) = w$ and $A(a,a) = f(a)$.
\end{proposition}

\begin{proof}
Let
$$S = \begin{bmatrix}
    f(a) & f(a) & a \\
    f(a) & w  & u \\
    a & u & 1 
\end{bmatrix}.$$
We have 
$$M=J_3^T S J_3 = \begin{bmatrix}
    f(a) & 0 & a-f(a) \\
    0 & w-f(a)  & u-a-w+f(a) \\
    a-f(a) & u-a-w+f(a) & w+1-2u 
\end{bmatrix}.$$
Note that, by Lagrange's theorem,
$$a-f(a) =a-u^*-(f(a)-f(u^*))=a-u^*-f'(\xi)(a-u^*) \ge 0,$$
where $\xi \in [u^*,a]$, since $f'(\xi)< 1$ by Lemma~\ref{function_f}.
Furthermore,
$$w-f(a) = f(u) - f(a) \ge 0,  \quad  u-a-w+f(a) \ge 0$$ 
by Lemma \ref{function_f}, and $w+1-2u \ge u^2+1-2u > 0$. Hence, the matrix $M$ is entry-wise nonnegative and $S$ is
a value matrix of some checkerboard copula $A$ with splitting points $\uD$ by Proposition \ref{prop:matrices} $(iii)$. 
The matrix $S$ is positive semidefinite, since
$$f(a)(w-f(a)) \ge 0, \quad f(a) - a^2 \ge 0, \quad w-u^2 \ge 0,$$
and
$$\det S = wf(a) + 2auf(a) - u^2f(a) - a^2w - f(a)^2 = 0, $$
because $f(a)$ is the larger zero of the quadratic equation $x^2 - (2au+w-u^2)x + a^2w =0$. Similarly as in the proof of Proposition \ref{prop:real_spodnja1} this implies that copula $A$ is of the form $A = \Psi(A_1)$ for some checkerboard copula $A_1$ by Theorem \ref{thm:cp-psi}, satisfying $A(a,a) = f(a)$ and $A(u,u) = w$.
\end{proof}

Our next proposition shows that the upper bound given by function $g$ is sharp on the interval $[u,u^{**}]$. We again explicitly construct a copula that attains it at a single point.

\begin{proposition} \label{prop:real_zgornja2}
    Let $u \in (0,1)$, $w \in [u^2, u]$,  $u^{**} = \tfrac{u^2}{w} \in [u,1]$ and $a \in (u, u^{**}]$.  Then there exists a checkerboard copula $A$ with splitting points $\uD = \delta_1 \times \delta_1$, with $\delta_1 = \{0, u, a, 1\} $, such that $A = \Psi(A_1)$ for some checkerboard copula $A_1$, $A(u,u) = w$ and $A(a,a) = g(a)$.
\end{proposition}

\begin{proof}
Similarly as in the proof of Proposition \ref{prop:real_spodnja2} let $\hat{u} = 1-u$, $\hat{a} = 1-a$, $\hat{w} = 1-2u+w$, and $u^* = \tfrac{\hat{w}-\hat{u}^2}{\hat{w}+1-2\hat{u}}$. Then $\hat{w} \in [\hat{u}^2, \hat{u}]$ and 
$ 1- u^*=u^{**},$ by equation~\eqref{eq:1-u*},
so $\hat{a} \in [u^*,\hat{u})$. Proposition~\ref{prop:real_zgornja1} implies there exists a checkerboard copula $B$ with splitting points $\widehat{\uD} = \delta_2' \times \delta_2'$, where $\delta_2' = \{0, \hat{a}, \hat{u}, 1\} $, such that $B = \Psi(B_1)$ for some checkerboard copula $B_1$,
$$B(\hat{a},\hat{a}) = \hat{a}^2+\tfrac14 \left(\sqrt{\hat{w}- (\hat{u}-2\hat{a})^2} + \sqrt{\hat{w}- \hat{u}^2}\right)^2$$ 
and $B(\hat{u},\hat{u}) = \hat{w}$. Let $A=\widehat{B}$. Then, by Lemma~\ref{lem:checkerboard_hat}, $A$ a checkerboard copula with splitting points $\uD = \delta_1 \times \delta_1$, where $\delta_1 = \{0, u, a, 1\} $, such that $A = \Psi(A_1)$ for some checkerboard copula $A_1$, $A(u,u) = w$ and $A(a,a) = 2a-1 + B(\hat{a},\hat{a}) = g(a)$.
\end{proof}

Finally, we show that at every point $a$ outside the interval $[u^*,u^{**}]$, a copula $C$ in the range of $\Psi$ that satisfies $C(u,u)=w$ can attain the global bound for the diagonal section, i.e. $C(a,a)=a$.

\begin{proposition} \label{prop:real_zgornja3}
    Let $u \in (0,1)$, $w \in [u^2, u]$,  $u^* = \tfrac{w-u^2}{w+1-2u} \in [0,u]$ and $u^{**} = \tfrac{u^2}{w} \in [u,1]$. Then there exists a copula $C$ such that $C = \Psi(C_1)$ for some copula $C_1$, $C(a,a) = a$ for every $a \in [0, u^*]$, and $C(u,u) = w$. Furthermore, there exists a copula $D$ such that $D = \Psi(D_1)$ for some copula $D_1$, $D(a,a) = a$ for every $a \in [u^{**},1]$, and $D(u,u) = w$.
\end{proposition}

\begin{proof} By Proposition \ref{prop:real_zgornja1} with $a=u^*$, there exists a checkerboard copula $A$ with splitting points $\uD = \delta_2 \times \delta_2$, where $\delta_2 = \{0, u^*, u, 1\} $, such that $A = \Psi(A_1)$ for some checkerboard copula $A_1$, $A(u^*,u^*) = f(u^*) = u^*$ and $A(u,u) = w$. Note that the value matrix of $A$ is 
$$S_A = \begin{bmatrix}
    u^* & u^* & u^* \\
    u^* & w  & u \\
    u^* & u & 1 
\end{bmatrix},$$
so $A(x, u^*) = A(u^*,x) = x$ for any $x \in [0, u^*]$.
We have
$$M_A=J_3^T S_A J_3=\begin{bmatrix}
    u^* & 0 & 0 \\
    0 & w-u^*  & u-w \\
    0 & u-w & 1+w-2u 
\end{bmatrix},$$
so that $G_A$ is of the form
$$G_A=\begin{bmatrix}
    g_{11} & 0 & 0 \\
    0 & g_{22}  & g_{23} \\
    0 & g_{32} & g_{33} 
\end{bmatrix}=G_{A_1}^TD_{\delta_1}G_{A_1}$$
by equation~\eqref{eq:G_psi}, implying that $A$ is an ordinal sum of two copulas with respect to $\{(0,u^*),(u^*,1)\}$.
Let $P$ be a permutation matrix such that the first column of $PG_{A_1}$ has all the zero elements collected at the end of the column, say, $(PG_{A_1})_{i1}$ is $0$ for $i>k$ and positive for $i \le k$.
Note that $PD_{\delta_1}P^T$ is again a diagonal matrix, say $PD_{\delta_1}P^T=D_{\delta_3}$, and $PG_{A_1}=G_{A_3}$ is again a density matrix of some checkerboard copula $A_3$ with splitting points $\delta_3 \times \delta_2$. Furthermore, $G_A=G_{A_3}^TD_{\delta_3}G_{A_3}$, so that $A=\Psi(A_3)$ by equation~\eqref{eq:G_psi}.
The zeros in $G_A$ imply that the first column of $G_{A_3}$ is orthogonal to the other two columns. Since the matrix entries are nonnegative, this implies $(G_{A_3})_{i2}=(G_{A_3})_{i3}=0$ for all $i\le k$. Consequently, $G_{A_3}$ and also $M_{A_3}$ is of the form
$$M_{A_3}=\begin{bmatrix}
    v_1 & 0 & 0 \\
    0 & v_2  & v_3 
\end{bmatrix},$$
for some vectors $v_1$, $v_2$, and $v_3$, where $v_1 \in \RR^k$. Note that the first column of $M_{A_3}$ corresponds to the interval $(0,u^*)$ and consequently, the first $k$ rows of $M_{A_3}$ also correspond to the interval $(0,u^*)$ since the sum of the components of $v_1$ equals $u^*$. This implies that $A_3$ is an ordinal sum of two checkerboard copulas, say $A_4$ and $A_5$, with respect to $\{(0,u^*),(u^*,1)\}$.
Now let $C_1$ be the ordinal sum of the Fr\'{e}chet-Hoeffding bound $M$ and $A_5$ with respect to $\{(0,u^*),(u^*,1)\}$.
By Proposition~\ref{prop:ordinal}, copula $A$ is an ordinal sum of $\Psi(A_4)$ and $\Psi(A_5)$ with respect to $\{(0,u^*),(u^*,1)\}$, while copula $C=\Psi(C_1)$ is an ordinal sum of $\Psi(M)$ and $\Psi(A_5)$ with respect to $\{(0,u^*),(u^*,1)\}$.
This implies that 
$$ C(x,y) = \begin{cases}
    \min\{x,y\}; &\text{if } x\le u^*, y\le u^*, \\
    A(x,y); &\text{otherwise. }
\end{cases}$$
Furthermore, $C(a,a) = a$ for every $a \in [0, u^*]$ and $C(u,u) = w$.

Similarly, by Proposition \ref{prop:real_zgornja2} with $a=u^{**}$, there there exists a checkerboard copula $B$ with splitting points $\uD = \delta_2' \times \delta_2'$, where $\delta_2' = \{0, u, u^{**}, 1\} $, such that $B = \Psi(B_1)$ for some checkerboard copula $B_1$, $B(u^{**},u^{**}) = g(u^{**}) = u^{**}$ and $B(u,u) = w$. We have  $B(x, u^{**}) = B(u^{**},x) = u^{**}$ for any $x \in [u^{**}, 1]$, implying that $B$ is an ordinal sum of two copulas with respect to $\{(0,u^{**}),(u^{**},1)\}$. If we let 
$$ D(x,y) = \begin{cases}
    \min\{x,y\}; &\text{if } x\ge u^{**}, y\ge u^{**}, \\
    B(x,y); &\text{otherwise, }
\end{cases}$$
then $D(a,a) = a$ for every $a \in [u^{**}, 1]$, $D(u,u) = w$, and, similarly as above, $D = \Psi(D_1)$ for some copula $D_1$.
\end{proof}

We collect our findings in our main theorem of this section.

\begin{theorem}
    Let $C$ be a copula such that $C = \Psi(C_1)$ for some copula $C_1$ and $u \in (0,1)$. Let $w = C(u,u) \in [u^2, u]$ be the value of $C$ at a diagonal point, let $u^* = \tfrac{w-u^2}{w+1-2u} \in [0,u]$ and $u^{**} = \tfrac{u^2}{w} \in [u,1]$. Then the upper bound for the diagonal section of $C$ is 
    $$C(a,a) \le \begin{cases}
			f(a); & \text{if }u^* \le a \le u, \\
			g(a); & \text{if }u \le a \le u^{**}, \\
			a; & \text{otherwise},
    \end{cases}$$
    where functions $f$ and $g$ are defined in equations \eqref{eq:f} and \eqref{eq:g}.
    Every point on the upper bound is attained by some copula $C$ of the form $C = \Psi(C_1)$ for some copula $C_1$ and $C(u,u) = w$.
\end{theorem}

\begin{proof}
We have $C(a,a) \le f(a)$ for any $a \in [0,u]$ by Proposition \ref{prop:zgornja1}, $C(a,a) \le g(a)$ for any $a \in [u,1]$ by Proposition \ref{prop:zgornja2}, and $C(a,a) \le a$ for any $a \in [0,1]$. Every point on the bound is attained by Propositions \ref{prop:real_zgornja1}, \ref{prop:real_zgornja2}, and \ref{prop:real_zgornja3}.
\end{proof}

We conclude the section with an example.

\begin{example}\label{ex:upper}
 Let $u = \tfrac25$ and $w=\tfrac15\in[\tfrac{4}{25},\tfrac25]$. Then $u^* = \tfrac{w-u^2}{w+1-2u} = \tfrac{1}{10}$ and $u^{**} = \tfrac{u^2}{w} = \tfrac45$. The upper bound for the diagonal section of any copula $C$ of the form $C = \Psi(C_1)$ for some copula $C_1$, satisfying the condition $C(u,u) = w$ is
    $$C(a,a) \le \begin{cases}
			\tfrac25 a + \tfrac{1}{50} + \tfrac{1}{50} \sqrt{-100a^2+40a+1}; & \text{if }\tfrac{1}{10} \le a \le \tfrac25, \\
			\tfrac75 a - \tfrac{19}{50} + \tfrac{1}{50} \sqrt{-100a^2+140a-39}; & \text{if }\tfrac25 \le a \le \tfrac45, \\
			a; & \text{otherwise}. 
    \end{cases}$$
The checkerboard copula $A$ of the form $A = \Psi(A_1)$ for some copula $A_1$ that attains the upper bound for $a_1=\tfrac15$ has splitting points $\uD = \delta_2 \times \delta_2$, where $\delta_2 = \{0, \tfrac15, \tfrac25, 1\}$ and the density matrix
$$G_A = \begin{bmatrix}
    3.618 & 0 & 0.461 \\
    0 & 1.382 & 1.206 \\
    0.461 & 1.206 & 1.111 
\end{bmatrix},$$
while the checkerboard copula $B$ of the form $B = \Psi(B_1)$ for some copula $B_1$ that attains the upper bound for $a_2=\tfrac35$ has splitting points $\uD' = \delta'_2 \times \delta'_2$, where $\delta'_2 = \{0, \tfrac25, \tfrac35, 1\}$ and the density matrix
$$G_B = \begin{bmatrix}
    \frac54 & \frac32 & \frac12 \\[1mm]
    \frac32 & 2 & 0 \\[1mm]
    \frac12 & 0 & 2
\end{bmatrix}.$$
Figure \ref{fig:upper} shows the graph of the upper bound together with diagonal sections of copulas $M$ and $\Pi$, diagonal sections of copulas $A$ and $B$, and their densities as heatmaps, where darker shades represent higher densities. 
\end{example}

\begin{figure}[ht]
\centering
\def\u{2/5}
\def\w{1/5}
\def\uu{1/10}
\def\uuu{4/5}
\def\a{1/5}
\def\b{3/5}
\begin{tikzpicture}[scale=4]
\draw[->] (-0.05,0) -- (1.1,0) node[right] {$a$};
\draw[->] (0,-0.05) -- (0,1.1);
\draw (\u,0.01) -- (\u,-0.01) node[below]{$u$};
\draw (\uu,0.01) -- (\uu,-0.01) node[below]{$u^*$};
\draw (\uuu,0.01) -- (\uuu,-0.01) node[below]{$u^{**}$};
\draw (1,0.01) -- (1,-0.01) node[below]{$1$};
\draw (0.01,\w) -- (-0.01,\w) node[left]{$w$};
\draw (0.01,\uu) -- (-0.01,\uu) node[left]{$u^*$};
\draw (0.01,\uuu) -- (-0.01,\uuu) node[left]{$u^{**}$};
\draw (0.01,1) -- (-0.01,1) node[left]{$1$};
\draw[-] (0,0) -- (1,1);
\draw[domain=0:1,samples=100] plot (\x,\x*\x);
\draw[black, very thick,domain=0:\uu,samples=100] plot (\x,\x);
\draw[black, very thick,domain=\uuu:1,samples=100] plot (\x,\x);
\draw[black, very thick,domain=\uu:\u,samples=100] plot (\x,{2/5*\x + 1/50 + 1/50*sqrt(-100*\x*\x+40*\x+1)});
\draw[black, very thick,domain=\u:\uuu,samples=100] plot (\x,{7/5*\x - 19/50 + 1/50*sqrt(-100*\x*\x+140*\x-39)});
\fill (\u,\w) circle (0.01) (\uu,\uu) circle (0.01) (\uuu,\uuu) circle (0.01);
\end{tikzpicture} 
\hspace{1cm}
\begin{tikzpicture}[scale=4]
\draw[->] (-0.05,0) -- (1.1,0) node[right] {$a$};
\draw[->] (0,-0.05) -- (0,1.1);
\draw (\u,0.01) -- (\u,-0.01) node[below]{$u$};
\draw (\uu,0.01) -- (\uu,-0.01) node[below]{$u^*$};
\draw (\uuu,0.01) -- (\uuu,-0.01) node[below]{$u^{**}$};
\draw (1,0.01) -- (1,-0.01) node[below]{$1$};
\draw (0.01,\w) -- (-0.01,\w) node[left]{$w$};
\draw (0.01,\uu) -- (-0.01,\uu) node[left]{$u^*$};
\draw (0.01,\uuu) -- (-0.01,\uuu) node[left]{$u^{**}$};
\draw (0.01,1) -- (-0.01,1) node[left]{$1$};
\draw[-] (0,0) -- (1,1);
\draw[domain=0:1,samples=100] plot (\x,\x*\x);
\draw[domain=0:\uu,samples=100] plot (\x,\x);
\draw[domain=\uuu:1,samples=100] plot (\x,\x);
\draw[domain=\uu:\u,samples=100] plot (\x,{2/5*\x + 1/50 + 1/50*sqrt(-100*\x*\x+40*\x+1)});
\draw[domain=\u:\uuu,samples=100] plot (\x,{7/5*\x - 19/50 + 1/50*sqrt(-100*\x*\x+140*\x-39)});
\draw[blue, very thick, domain=0:\a,samples=100] plot (\x,3.618*\x*\x);
\draw[blue, very thick, domain=\a:\u,samples=100] plot (\x,1.382*\x*\x-0.553*\x+0.2);
\draw[blue, very thick, domain=\u:1,samples=100] plot (\x,10/9*\x*\x-2/9*\x+1/9);
\fill[blue] (\a,3.618*\a*\a) circle (0.01);
\fill[blue] (\u,\w) circle (0.01);
\draw[red, very thick, domain=0:\u,samples=100] plot (\x,5/4*\x*\x);
\draw[red, very thick, domain=\u:\b,samples=100] plot (\x,2*\x*\x-2/5*\x+1/25);
\draw[red, very thick, domain=\b:1,samples=100] plot (\x,2*\x*\x-2*\x+1);
\fill[red] (\b,2*\b*\b-2*\b+1) circle (0.01);
\fill[red] (\u,\w) circle (0.01);
\end{tikzpicture}
\\[5mm]
\begin{tikzpicture}[scale=4]
\filldraw[black!56] (0,0) -- (0,\a) -- (\a,\a) -- (\a,0) -- cycle;
\filldraw[black!24] (\a,\a) -- (\u,\a) -- (\u,\u) -- (\a,\u) -- cycle;
\filldraw[black!7] (0,\u) -- (0,1) -- (\a,1) -- (\a,\u) -- cycle;
\filldraw[black!7] (\u,0) -- (1,0) -- (1,\a) -- (\u,\a) -- cycle;
\filldraw[black!19] (\a,\u) -- (\a,1) -- (\u,1) -- (\u,\u) -- cycle;
\filldraw[black!19] (\u,\a) -- (1,\a) -- (1,\u) -- (\u,\u) -- cycle;
\filldraw[black!17] (\u,\u) -- (1,\u) -- (1,1) -- (\u,1) -- cycle;
\draw[thick] (0,0) -- (1,0) -- (1,1) -- (0,1) -- cycle;
\draw (\a,1) -- (\a,0) node[below]{$a_1$};
\draw (\u,1) -- (\u,0) node[below]{$u$};
\draw (1,1) -- (1,0) node[below]{$1$};
\draw (1,\a) -- (0,\a) node[left]{$a_1$};
\draw (1,\u) -- (0,\u) node[left]{$u$};
\draw (1,1) -- (0,1) node[left]{$1$};
\end{tikzpicture}
\hspace{2cm}
\begin{tikzpicture}[scale=4]
\filldraw[black!20] (0,0) -- (0,\u) -- (\u,\u) -- (\u,0) -- cycle;
\filldraw[black!24] (0,\u) -- (0,\b) -- (\u,\b) -- (\u,\u) -- cycle;
\filldraw[black!24] (\u,0) -- (\b,0) -- (\b,\u) -- (\u,\u) -- cycle;
\filldraw[black!32] (\u,\u) -- (\u,\b) -- (\b,\b) -- (\b,\u) -- cycle;
\filldraw[black!32] (1,1) -- (1,\b) -- (\b,\b) -- (\b,1) -- cycle;
\filldraw[black!8] (\u,\b) -- (\u,1) -- (0,1) -- (0,\b) -- cycle;
\filldraw[black!8] (\b,\u) -- (1,\u) -- (1,0) -- (\b,0) -- cycle;
\draw[thick] (0,0) -- (1,0) -- (1,1) -- (0,1) -- cycle;
\draw (\u,1) -- (\u,0) node[below]{$u$};
\draw (\b,1) -- (\b,0) node[below]{$a_2$};
\draw (1,1) -- (1,0) node[below]{$1$};
\draw (1,\u) -- (0,\u) node[left]{$u$};
\draw (1,\b) -- (0,\b) node[left]{$a_2$};
\draw (1,1) -- (0,1) node[left]{$1$};
\end{tikzpicture}
    \caption{The graph of the upper bound together with diagonal sections of copulas $M$ and $\Pi$ (top left), diagonal sections of copulas $A$ in blue and $B$ in red (top right) from Example \ref{ex:upper}, and densities of copulas $A$ (bottom left) and $B$ (bottom right) as heatmaps.}
    \label{fig:upper}
\end{figure}
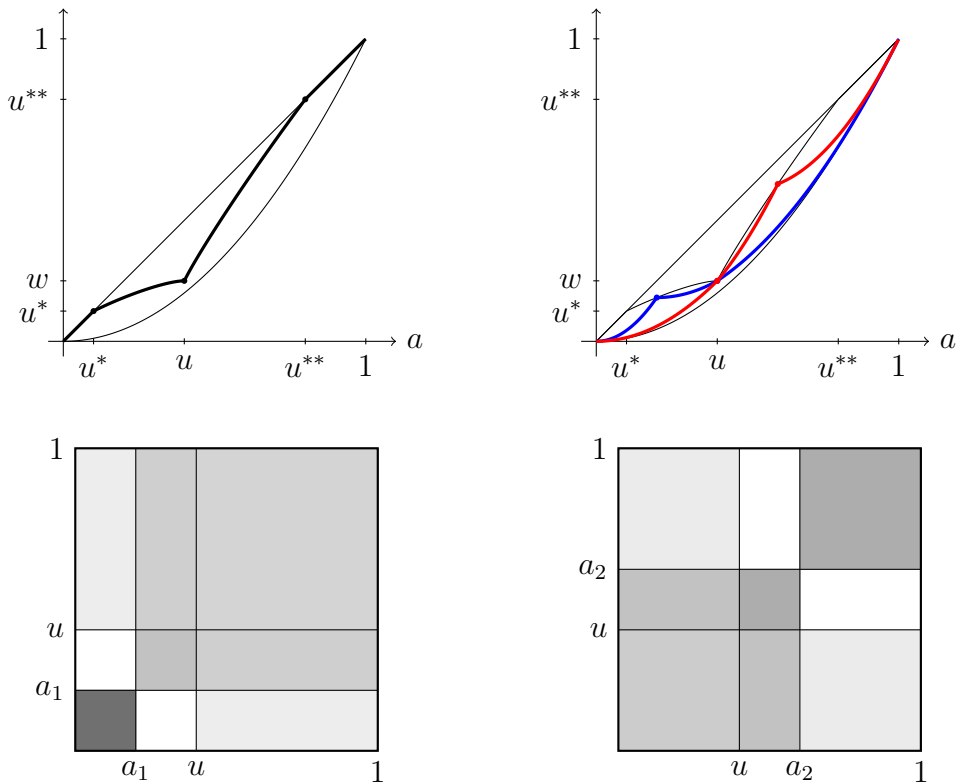

\section{Concluding remarks} \label{sec:concluding}

We studied the copulas in the range of the transformation $\Psi(C)=C^T*C$ on the class of bivariate copulas. We focused on diagonal sections of copulas in the range and on checkerboard copulas. We determined the exact lower and upper bound for the diagonal section of copula $\Psi(C)$ under the condition that its value is known in a single point. Both bounds are piecewise functions.
The pieces of the lower bound are quadratic functions, while the crucial pieces of the upper bound are square-root functions. We constructed copulas that attain the bounds and observed drastically different behaviour for the two bounds; the whole piece of the lower bound is attained by a single copula, while the crucial pieces of the upper bound are only attained point-by-point.

We remark that our result cannot be directly extended to the case when the values of the diagonal section of $\Psi(C)$ are known at two or more points. Namely, suppose that $0<a<b<1$ and the value of the diagonal section of $\Psi(C)$ is known at the points $a$ and $b$ (and these values respect all the bounds obtained in our results).
Let $\varphi_a(u)$ be the lower bound for the diagonal section given the value at $a$ given by Theorem~\ref{thm:main_lower}, and similarly, let $\varphi_b(u)$ be the lower bound given the value at $b$. Then the lower bound for the diagonal section of $\Psi(C)$ given the values at both $a$ and $b$ is not equal to $\max\{\varphi_a(u),\varphi_b(u)\}$.
One of the reasons for this is the fact that a $4 \times 4$ matrix whose two $3 \times 3$ submatrices are positive semidefinite need not be positive semidefinite itself.

In the future it would be interesting to study the range of $\Psi$ in more detail from statistical and topological point of view. Also, the two open implications in Figure~\ref{fig:results} could be investigated.

\section*{Acknowledgments}

The authors acknowledge financial support from the ARIS (Slovenian Research and Innovation Agency), research core funding No. P1-0222 (Roman Drnovšek, Damjana Kokol Bukov\v{s}ek, Nik Stopar), and project J1-70034 (Damjana Kokol Bukov\v{s}ek, Nik Stopar).

\section*{Declaration of generative AI and AI-assisted technologies}

No AI tools were used during the preparation of this manuscript.

\bibliographystyle{amsplain}
\bibliography{biblio}

@article {DarNguOls92,
    AUTHOR = {Darsow, W.F. and Nguyen, B. and Olsen, E.T.},
     TITLE = {Copulas and {M}arkov processes},
   JOURNAL = {Illinois J. Math.},
  FJOURNAL = {Illinois Journal of Mathematics},
    VOLUME = {36},
      YEAR = {1992},
    NUMBER = {4},
     PAGES = {600--642},
      ISSN = {0019-2082,1945-6581},
   MRCLASS = {60J99 (46N30)},
  MRNUMBER = {1215798},
MRREVIEWER = {B.\ Schweizer},
       URL = {http://projecteuclid.org/euclid.ijm/1255987328},
}

@book {Nel06,
    AUTHOR = {Nelsen, R.B.},
     TITLE = {An introduction to copulas},
    SERIES = {Springer Series in Statistics},
   EDITION = {Second},
 PUBLISHER = {Springer, New York},
      YEAR = {2006},
     PAGES = {xiv+269},
      ISBN = {978-0387-28659-4; 0-387-28659-4},
   MRCLASS = {60-02 (60E05 60E15 62H05)},
  MRNUMBER = {2197664},
MRREVIEWER = {Radu\ Theodorescu},
       DOI = {10.1007/s11229-005-3715-x},
       URL = {https://doi.org/10.1007/s11229-005-3715-x},
}

@incollection {FrNe97,
    AUTHOR = {Fredricks, G.A. and Nelsen, R.B.},
     TITLE = {Copulas constructed from diagonal sections},
 BOOKTITLE = {Distributions with given marginals and moment problems
              ({P}rague, 1996)},
     PAGES = {129--136},
 PUBLISHER = {Kluwer Acad. Publ., Dordrecht},
      YEAR = {1997},
      ISBN = {0-7923-4573-8},
   MRCLASS = {60E05 (62H05)},
  MRNUMBER = {1614666},
}

@article {Skl59,
    AUTHOR = {Sklar, A.},
     TITLE = {Fonctions de r\'epartition \`a{} {$n$} dimensions et leurs
              marges},
   JOURNAL = {Publ. Inst. Statist. Univ. Paris},
  FJOURNAL = {Publications de l'Institut de Statistique de l'Universit\'e{}
              de Paris},
    VOLUME = {8},
      YEAR = {1959},
     PAGES = {229--231},
   MRCLASS = {60.20},
  MRNUMBER = {125600},
MRREVIEWER = {M.\ Lo\`eve},
}

@article {LinWanZhaZha25,
    AUTHOR = {Lin, L. and Wang, R. and Zhang, R. and Zhao,
              C.},
     TITLE = {The checkerboard copula and dependence concepts},
   JOURNAL = {SIAM J. Financial Math.},
  FJOURNAL = {SIAM Journal on Financial Mathematics},
    VOLUME = {16},
      YEAR = {2025},
    NUMBER = {2},
     PAGES = {426--447},
      ISSN = {1945-497X},
   MRCLASS = {62H05 (91B05 91G70)},
  MRNUMBER = {4897917},
MRREVIEWER = {Pranesh\ Kumar},
       DOI = {10.1137/24M1656645},
       URL = {https://doi.org/10.1137/24M1656645},
}

@article {MikTay10,
    AUTHOR = {Mikusi\'nski, P. and Taylor, M.D.},
     TITLE = {Some approximations of {$n$}-copulas},
   JOURNAL = {Metrika},
  FJOURNAL = {Metrika. International Journal for Theoretical and Applied
              Statistics},
    VOLUME = {72},
      YEAR = {2010},
    NUMBER = {3},
     PAGES = {385--414},
      ISSN = {0026-1335,1435-926X},
   MRCLASS = {62H05 (46N30 47B38 60B10 60E05 62E17)},
  MRNUMBER = {2746583},
MRREVIEWER = {Roger\ B.\ Nelsen},
       DOI = {10.1007/s00184-009-0259-y},
       URL = {https://doi.org/10.1007/s00184-009-0259-y},
}

@article {JuGrTr21,
    AUTHOR = {Junker, R.R. and Griessenberger, F. and Trutschnig,
              W.},
     TITLE = {Estimating scale-invariant directed dependence of bivariate
              distributions},
   JOURNAL = {Comput. Statist. Data Anal.},
  FJOURNAL = {Computational Statistics \& Data Analysis},
    VOLUME = {153},
      YEAR = {2021},
     PAGES = {Paper No. 107058, 22},
      ISSN = {0167-9473,1872-7352},
   MRCLASS = {62H05 (62G05)},
  MRNUMBER = {4141460},
MRREVIEWER = {Gilles\ Stupfler},
       DOI = {10.1016/j.csda.2020.107058},
       URL = {https://doi.org/10.1016/j.csda.2020.107058},
}

@article {KoPe25,
    AUTHOR = {Ko\v{s}ir, T. and Perrone, E.},
     TITLE = {Discrete imprecise copulas},
   JOURNAL = {Fuzzy Sets and Systems},
  FJOURNAL = {Fuzzy Sets and Systems. An International Journal in
              Information Science and Engineering},
    VOLUME = {504},
      YEAR = {2025},
     PAGES = {Paper No. 109251, 20},
      ISSN = {0165-0114,1872-6801},
   MRCLASS = {99-01},
  MRNUMBER = {4843882},
       DOI = {10.1016/j.fss.2024.109251},
       URL = {https://doi.org/10.1016/j.fss.2024.109251},
}

@article {Fu24,
    AUTHOR = {Fuchs, S.},
     TITLE = {Quantifying directed dependence via dimension reduction},
   JOURNAL = {J. Multivariate Anal.},
  FJOURNAL = {Journal of Multivariate Analysis},
    VOLUME = {201},
      YEAR = {2024},
     PAGES = {Paper No. 105266, 21},
      ISSN = {0047-259X,1095-7243},
   MRCLASS = {62H20 (62E10)},
  MRNUMBER = {4714790},
       DOI = {10.1016/j.jmva.2023.105266},
       URL = {https://doi.org/10.1016/j.jmva.2023.105266},
}

@book {BeShMo03,
    AUTHOR = {Berman, A. and Shaked-Monderer, N.},
     TITLE = {Completely positive matrices},
 PUBLISHER = {World Scientific Publishing Co., Inc., River Edge, NJ},
      YEAR = {2003},
     PAGES = {x+206},
      ISBN = {981-238-368-9},
   MRCLASS = {15A48 (15-02)},
  MRNUMBER = {1986666},
MRREVIEWER = {Aurelian Gheondea},
       DOI = {10.1142/9789812795212},
       URL = {https://doi.org/10.1142/9789812795212},
}

@article {Ch21,
    AUTHOR = {Chatterjee, S.},
     TITLE = {A new coefficient of correlation},
   JOURNAL = {J. Amer. Statist. Assoc.},
  FJOURNAL = {Journal of the American Statistical Association},
    VOLUME = {116},
      YEAR = {2021},
    NUMBER = {536},
     PAGES = {2009--2022},
      ISSN = {0162-1459,1537-274X},
   MRCLASS = {99-01},
  MRNUMBER = {4353729},
       DOI = {10.1080/01621459.2020.1758115},
       URL = {https://doi.org/10.1080/01621459.2020.1758115},
}

@article {AzCh21,
    AUTHOR = {Azadkia, M. and Chatterjee, S.},
     TITLE = {A simple measure of conditional dependence},
   JOURNAL = {Ann. Statist.},
  FJOURNAL = {The Annals of Statistics},
    VOLUME = {49},
      YEAR = {2021},
    NUMBER = {6},
     PAGES = {3070--3102},
      ISSN = {0090-5364,2168-8966},
   MRCLASS = {62G05 (62H20)},
  MRNUMBER = {4352523},
       DOI = {10.1214/21-aos2073},
       URL = {https://doi.org/10.1214/21-aos2073},
}

@article {DeSiSt13,
    AUTHOR = {Dette, H. and Siburg, K.F. and Stoimenov, P.A.},
     TITLE = {A copula-based non-parametric measure of regression
              dependence},
   JOURNAL = {Scand. J. Stat.},
  FJOURNAL = {Scandinavian Journal of Statistics. Theory and Applications},
    VOLUME = {40},
      YEAR = {2013},
    NUMBER = {1},
     PAGES = {21--41},
      ISSN = {0303-6898,1467-9469},
   MRCLASS = {62H20 (60E15 62G08 62G20)},
  MRNUMBER = {3024030},
       DOI = {10.1111/j.1467-9469.2011.00767.x},
       URL = {https://doi.org/10.1111/j.1467-9469.2011.00767.x},
}

@article {FuLi26,
    AUTHOR = {Fuchs, S. and Limbach, C.},
     TITLE = {A dimension reduction for extreme types of directed
              dependence},
   JOURNAL = {Depend. Model.},
  FJOURNAL = {Dependence Modeling},
    VOLUME = {14},
      YEAR = {2026},
    NUMBER = {1},
     PAGES = {Paper No. 20250016, 14},
      ISSN = {2300-2298},
   MRCLASS = {62H20 (62E10)},
  MRNUMBER = {5025019},
       DOI = {10.1515/demo-2025-0016},
       URL = {https://doi.org/10.1515/demo-2025-0016},
}

@article {AnRo24,
    AUTHOR = {Ansari, J. and Rockel, M.},
     TITLE = {Dependence properties of bivariate copula families},
   JOURNAL = {Depend. Model.},
  FJOURNAL = {Dependence Modeling},
    VOLUME = {12},
      YEAR = {2024},
    NUMBER = {1},
     PAGES = {Paper No. 20240002, 36},
      ISSN = {2300-2298},
   MRCLASS = {62H05 (60E15 62H20)},
  MRNUMBER = {4778718},
       DOI = {10.1515/demo-2024-0002},
       URL = {https://doi.org/10.1515/demo-2024-0002},
}

@article {JaHuMi17,
    AUTHOR = {Jaser, M. and Haug, S. and Min, A.},
     TITLE = {A simple non-parametric goodness-of-fit test for elliptical
              copulas},
   JOURNAL = {Depend. Model.},
  FJOURNAL = {Dependence Modeling},
    VOLUME = {5},
      YEAR = {2017},
    NUMBER = {1},
     PAGES = {330--353},
      ISSN = {2300-2298},
   MRCLASS = {62G10 (62H20)},
  MRNUMBER = {3745440},
MRREVIEWER = {Jean-Fran\c cois\ Quessy},
       DOI = {10.1515/demo-2017-0020},
       URL = {https://doi.org/10.1515/demo-2017-0020},
}

@article {JoLiNi10,
    AUTHOR = {Joe, H. and Li, H. and Nikoloulopoulos, A.K.},
     TITLE = {Tail dependence functions and vine copulas},
   JOURNAL = {J. Multivariate Anal.},
  FJOURNAL = {Journal of Multivariate Analysis},
    VOLUME = {101},
      YEAR = {2010},
    NUMBER = {1},
     PAGES = {252--270},
      ISSN = {0047-259X,1095-7243},
   MRCLASS = {62H05 (60G70 62H20)},
  MRNUMBER = {2557632},
MRREVIEWER = {Fabrizio\ Durante},
       DOI = {10.1016/j.jmva.2009.08.002},
       URL = {https://doi.org/10.1016/j.jmva.2009.08.002},
}

@article {GuSe12,
    AUTHOR = {Gudendorf, G. and Segers, J.},
     TITLE = {Nonparametric estimation of multivariate extreme-value
              copulas},
   JOURNAL = {J. Statist. Plann. Inference},
  FJOURNAL = {Journal of Statistical Planning and Inference},
    VOLUME = {142},
      YEAR = {2012},
    NUMBER = {12},
     PAGES = {3073--3085},
      ISSN = {0378-3758,1873-1171},
   MRCLASS = {62H12 (62G05 62G32 62H05)},
  MRNUMBER = {2956794},
       DOI = {10.1016/j.jspi.2012.05.007},
       URL = {https://doi.org/10.1016/j.jspi.2012.05.007},
}

@article {GeNeZi11,
    AUTHOR = {Genest, C. and Ne\v{s}lehov\'{a}, J. and Ziegel,
              J.},
     TITLE = {Inference in multivariate {A}rchimedean copula models},
   JOURNAL = {TEST},
  FJOURNAL = {TEST. An Official Journal of the Spanish Society of Statistics
              and Operations Research},
    VOLUME = {20},
      YEAR = {2011},
    NUMBER = {2},
     PAGES = {223--256},
      ISSN = {1133-0686,1863-8260},
   MRCLASS = {62G05 (62G20 62H12 62H20)},
  MRNUMBER = {2834039},
MRREVIEWER = {Xiaohu\ Li},
       DOI = {10.1007/s11749-011-0250-6},
       URL = {https://doi.org/10.1007/s11749-011-0250-6},
}

@article {ScPaTr17,
    AUTHOR = {Schreyer, M. and Paulin, R. and Trutschnig, W.},
     TITLE = {On the exact region determined by {K}endall's {$\tau$} and
              {S}pearman's {$\rho$}},
   JOURNAL = {J. R. Stat. Soc. Ser. B. Stat. Methodol.},
  FJOURNAL = {Journal of the Royal Statistical Society. Series B.
              Statistical Methodology},
    VOLUME = {79},
      YEAR = {2017},
    NUMBER = {2},
     PAGES = {613--633},
      ISSN = {1369-7412,1467-9868},
   MRCLASS = {62H20 (60C05)},
  MRNUMBER = {3611762},
MRREVIEWER = {Carlo\ E.\ Sempi},
       DOI = {10.1111/rssb.12181},
       URL = {https://doi.org/10.1111/rssb.12181},
}

@article {GeNeGh10,
    AUTHOR = {Genest, C. and Ne\v{s}lehov\'{a}, J. and Ben Ghorbal,
              N.},
     TITLE = {Spearman's footrule and {G}ini's gamma: a review with
              complements},
   JOURNAL = {J. Nonparametr. Stat.},
  FJOURNAL = {Journal of Nonparametric Statistics},
    VOLUME = {22},
      YEAR = {2010},
    NUMBER = {8},
     PAGES = {937--954},
      ISSN = {1048-5252},
   MRCLASS = {62-02 (62G10 62G20 62G30 62H20)},
  MRNUMBER = {2738876},
       DOI = {10.1080/10485250903499667},
       URL = {https://doi.org/10.1080/10485250903499667},
}

@article {Wy15,
    AUTHOR = {Wysocki, W.},
     TITLE = {{K}endall's tau and {S}pearman's rho for {$n$}-dimensional
              {A}rchimedean copulas and their asymptotic properties},
   JOURNAL = {J. Nonparametr. Stat.},
  FJOURNAL = {Journal of Nonparametric Statistics},
    VOLUME = {27},
      YEAR = {2015},
    NUMBER = {4},
     PAGES = {442--459},
      ISSN = {1048-5252},
   MRCLASS = {62H20 (62G20 62G30)},
  MRNUMBER = {3406321},
       DOI = {10.1080/10485252.2015.1070849},
       URL = {https://doi.org/10.1080/10485252.2015.1070849},
}

@article {DoUbFl06,
    AUTHOR = {Dolati, A. and \'{U}beda-Flores, M.},
     TITLE = {On measures on multivariate concordance},
   JOURNAL = {JPSS J. Probab. Stat. Sci.},
  FJOURNAL = {JPSS. Journal of Probability and Statistical Science},
    VOLUME = {4},
      YEAR = {2006},
    NUMBER = {2},
     PAGES = {147--163},
      ISSN = {1726-3328},
   MRCLASS = {62H20},
  MRNUMBER = {2488161},
}

@article {UbFl05,
    AUTHOR = {\'{U}beda-Flores, M.},
     TITLE = {Multivariate versions of {B}lomqvist's beta and {S}pearman's footrule},
   JOURNAL = {Ann. Inst. Statist. Math.},
  FJOURNAL = {Ann. Inst. Statist. Math.},
    VOLUME = {57},
      YEAR = {2005},
    NUMBER = {4},
     PAGES = {781--788},
}

@article {FuSc21,
    AUTHOR = {Fuchs, S. and Schmidt, K.D.},
     TITLE = {On order statistics and {K}endall's tau},
   JOURNAL = {Statist. Probab. Lett.},
  FJOURNAL = {Statistics \& Probability Letters},
    VOLUME = {169},
      YEAR = {2021},
     PAGES = {Paper No. 108972, 7},
      ISSN = {0167-7152},
   MRCLASS = {62G30 (62H20)},
  MRNUMBER = {4168327},
       DOI = {10.1016/j.spl.2020.108972},
       URL = {https://doi.org/10.1016/j.spl.2020.108972},
}

@article{Li14,
author = {Liebscher, L.},
doi = {doi:10.2478/demo-2014-0004},
url = {https://doi.org/10.2478/demo-2014-0004},
title = {Copula-based dependence measures},
journal = {Dependence Modeling},
number = {1},
volume = {2},
year = {2014},
pages = {49--64},
lastchecked = {2022-05-26}
}

@article {Sc84,
    AUTHOR = {Scarsini, M.},
     TITLE = {On measures of concordance},
   JOURNAL = {Stochastica},
  FJOURNAL = {Stochastica},
    VOLUME = {8},
      YEAR = {1984},
    NUMBER = {3},
     PAGES = {201--218},
      ISSN = {0210-7821},
   MRCLASS = {62H20 (60E05)},
  MRNUMBER = {796650},
MRREVIEWER = {B. Schweizer},
}

@article{DaEA11,
author = {Reshef, D.N. and Reshef, Y.A. and Finucane, H.K. and Grossman, S.R. and McVean, G. and Turnbaugh, P.J. and Lander, E.S. and Mitzenmacher, M. and Sabeti, P.C.},
title = {Detecting Novel Associations in Large Data Sets},
journal = {Science},
volume = {334},
number = {6062},
pages = {1518--1524},
year = {2011},
doi = {10.1126/science.1205438},
URL = {https://www.science.org/doi/abs/10.1126/science.1205438},
eprint = {https://www.science.org/doi/pdf/10.1126/science.1205438}}

@article {ScWo81,
    AUTHOR = {Schweizer, B. and Wolff, E.F.},
     TITLE = {On nonparametric measures of dependence for random variables},
   JOURNAL = {Ann. Statist.},
  FJOURNAL = {The Annals of Statistics},
    VOLUME = {9},
      YEAR = {1981},
    NUMBER = {4},
     PAGES = {879--885},
      ISSN = {0090-5364,2168-8966},
   MRCLASS = {62H20 (62G99)},
  MRNUMBER = {619291},
MRREVIEWER = {D.\ M.\ D\polhk abrowska},
       URL =
              {http://links.jstor.org/sici?sici=0090-5364(198107)9:4<879:ONMODF>2.0.CO;2-2&origin=MSN},
}

@article {Re59,
    AUTHOR = {R\'{e}nyi, A.},
     TITLE = {On measures of dependence},
   JOURNAL = {Acta Math. Acad. Sci. Hungar.},
  FJOURNAL = {Acta Mathematica. Academiae Scientiarum Hungaricae},
    VOLUME = {10},
      YEAR = {1959},
     PAGES = {441--451 (unbound insert)},
      ISSN = {0001-5954,1588-2632},
   MRCLASS = {60.00 (94.00)},
  MRNUMBER = {115203},
MRREVIEWER = {H.\ P.\ Kramer},
       DOI = {10.1007/BF02024507},
       URL = {https://doi.org/10.1007/BF02024507},
}

@article {LeTa96,
    AUTHOR = {Ledford, A.W. and Tawn, J.A.},
     TITLE = {Statistics for near independence in multivariate extreme
              values},
   JOURNAL = {Biometrika},
  FJOURNAL = {Biometrika},
    VOLUME = {83},
      YEAR = {1996},
    NUMBER = {1},
     PAGES = {169--187},
      ISSN = {0006-3444,1464-3510},
   MRCLASS = {62G20 (60G70 62H99)},
  MRNUMBER = {1399163},
MRREVIEWER = {Margaret\ P.\ Gessaman},
       DOI = {10.1093/biomet/83.1.169},
       URL = {https://doi.org/10.1093/biomet/83.1.169},
}

@article {HuJo14,
    AUTHOR = {Hua, L. and Joe, H.},
     TITLE = {Strength of tail dependence based on conditional tail
              expectation},
   JOURNAL = {J. Multivariate Anal.},
  FJOURNAL = {Journal of Multivariate Analysis},
    VOLUME = {123},
      YEAR = {2014},
     PAGES = {143--159},
      ISSN = {0047-259X,1095-7243},
   MRCLASS = {62H05 (62G32 62H20)},
  MRNUMBER = {3130426},
MRREVIEWER = {Chao\ Huang},
       DOI = {10.1016/j.jmva.2013.09.001},
       URL = {https://doi.org/10.1016/j.jmva.2013.09.001},
}

@article {ScSc07,
    AUTHOR = {Schmid, F. and Schmidt, R.},
     TITLE = {Multivariate conditional versions of {S}pearman's rho and
              related measures of tail dependence},
   JOURNAL = {J. Multivariate Anal.},
  FJOURNAL = {Journal of Multivariate Analysis},
    VOLUME = {98},
      YEAR = {2007},
    NUMBER = {6},
     PAGES = {1123--1140},
      ISSN = {0047-259X,1095-7243},
   MRCLASS = {62H20 (60F05 60G15 60G70 62G20)},
  MRNUMBER = {2326243},
MRREVIEWER = {Winfried\ Stute},
       DOI = {10.1016/j.jmva.2006.05.005},
       URL = {https://doi.org/10.1016/j.jmva.2006.05.005},
}

@article {DuFePa15,
    AUTHOR = {Durante, F. and Fern\'{a}ndez-S\'{a}nchez, J. and
              Pappad\`{a}, R.},
     TITLE = {Copulas, diagonals, and tail dependence},
   JOURNAL = {Fuzzy Sets and Systems},
  FJOURNAL = {Fuzzy Sets and Systems. An International Journal in
              Information Science and Engineering},
    VOLUME = {264},
      YEAR = {2015},
     PAGES = {22--41},
      ISSN = {0165-0114,1872-6801},
   MRCLASS = {62H05 (62-09 62G32)},
  MRNUMBER = {3303661},
       DOI = {10.1016/j.fss.2014.03.014},
       URL = {https://doi.org/10.1016/j.fss.2014.03.014},
}

\end{document}